\documentclass[11pt]{article}
\usepackage[a4paper,margin=25mm]{geometry}
\usepackage{lmodern}
\usepackage{xcolor,microtype,mathtools,amssymb,amsthm}
\usepackage{xurl}
\usepackage[colorlinks=true,linkcolor=blue!55!black,
  citecolor=blue!55!black,urlcolor=blue!55!black,
  bookmarksnumbered=true,bookmarksopen=true,bookmarksopenlevel=2,
  pdfpagemode=UseOutlines]{hyperref}

\numberwithin{equation}{section}
\newtheorem{theorem}{Theorem}[section]
\newtheorem{proposition}[theorem]{Proposition}
\newtheorem{lemma}[theorem]{Lemma}
\newtheorem{remark}[theorem]{Remark}

\newcommand{\Z}{\mathbb Z}
\newcommand{\C}{\mathbb C}
\newcommand{\R}{\mathbb R}
\newcommand{\T}{\mathbb T}
\newcommand{\N}{\mathbb N}
\newcommand{\e}{\mathrm e}
\newcommand{\es}[1]{\mathrm e(#1)}
\newcommand{\dd}{\,\mathrm d}
\newcommand{\dist}{\operatorname{dist}}
\newcommand{\Rea}{\operatorname{Re}}
\newcommand{\Ima}{\operatorname{Im}}
\newcommand{\norm}[1]{\lVert #1\rVert}
\newcommand{\abs}[1]{\lvert #1\rvert}

\title{\Large\bf Dimension-free estimates for the full discrete Euclidean ball maximal function}
\author{Sheng-Chen Mao}
\date{}

\begin{document}

\maketitle

\noindent\textbf{Abstract}. 
Let $M_t$ denote the normalized average over the lattice points in the
Euclidean ball of radius $t$ in $\mathbb{Z}^d$.  We prove that the full maximal
operator $f\mapsto\sup_{t\geq0}\lvert M_t f\rvert$ is bounded on
$\ell^p(\mathbb{Z}^d)$, for every $1<p\leq\infty$, with a constant independent
of the dimension.  In particular, this resolves a question of E.M. Stein
from the mid 1990s.  The principal ingredient in our proof is that, when
$t\lesssim d$ with $t$ sufficiently large, the associated multiplier
$\mathfrak{m}_{\sqrt{\lfloor t^2\rfloor}}(\xi)$ admits an asymptotic expansion
of arbitrary prescribed order, uniform in $\xi$, whose resulting maximal
operators can be controlled by the discrete normalized Gaussian maximal
function studied by Mirek--Szarek--Wr\'obel \cite{MSW25}.



\section{Introduction}\label{sint}

Dimension-free maximal inequalities ask whether a family of averages admits
operator bounds that remain uniform as the ambient dimension increases.  In
a fixed dimension, covering arguments may introduce constants that depend on
the dimension; a uniform estimate must instead use the geometry and Fourier
structure of the averaging sets.  Euclidean balls provide the fundamental
positive example in the continuous setting.  Existing
small-radius and large-radius methods leave an intermediate interval.  We
prove the required estimate in this interval and thereby obtain a
dimension-free bound for every \(1<p\leq\infty\).

\subsection{Main results}\label{smain}

For \(d\in\N\), let
\begin{equation*}
 B:=B_2(d):=\{x\in\R^d:\abs{x}\leq1\},\qquad
 tB:=tB_2(d)\cap\Z^d\quad(t\geq0),
\end{equation*}
where \(\abs{x}\) denotes the  Euclidean $\ell^2$-norm.  Following Niksi\'nski and
Wr\'obel~\cite[Section 1.5]{NW26}, we use \(tB\) for the lattice set
\(tB_2(d)\cap\Z^d\), rather than for the corresponding subset of
\(\R^d\).  The normalized discrete ball average is
\begin{equation*}
 M_tf(x):=\frac{1}{\#(tB)}
 \sum_{y\in tB}f(x-y),
 \qquad f\in \ell^1(\mathbb{Z}^d),\quad x\in\Z^d.
\end{equation*}
Each \(M_t\) is a contraction on \(\ell^p(\Z^d)\), so a single radius causes
no loss depending on \(d\).  Classical weak type \((1,1)\) arguments and
interpolation bound the supremum in each fixed dimension, but the resulting
constants grow exponentially with \(d\); see Niksi\'nski and
Wr\'obel~\cite[Section 1.1]{NW26}.  Stein posed the problem of avoiding this
dimensional loss in the mid 1990s.  More precisely, one asks whether
\begin{equation}\label{e115}
 \sup_{d\in\N}\sup_{f\neq0}
 \frac{\norm{\sup_{t\geq0}\abs{M_tf}}_{\ell^2(\Z^d)}}
 {\norm{f}_{\ell^2(\Z^d)}}\leq C.
\end{equation}

We prove the following strong-type estimate, which gives an affirmative answer
to \eqref{e115} and covers every \(1<p\leq\infty\).

\begin{theorem}\label{mthm}
Let \(p\in(1,\infty]\). There is a constant \(C_p>0\)  independent of \(d\)
such that, for every \(d\in\N\) and every \(f\in\ell^p(\Z^d)\),
\begin{equation}\label{e004}
 \norm{\sup_{t\geq0}\, \abs{M_tf}}_{\ell^p(\Z^d)}
 \leq C_p\, \norm{f}_{\ell^p(\Z^d)}.
\end{equation}
\end{theorem}

\begin{remark}
While this manuscript was being finalized, two closely related preprints
appeared.  Jin and Su~\cite[Theorem 1.1]{JS26} and Hormozi,
Niksi\'nski, and Wr\'obel~\cite[Theorem 1]{HNW26} prove the same
full-radius dimension-free estimate for discrete Euclidean balls as
Theorem~\ref{mthm}.  The paper \cite[Theorem 2]{HNW26} also proves a dimension-free estimate for the
full discrete spherical maximal operator when $d\geq5$ and
$2\leq p\leq\infty$.  We will compare the three approaches in
Subsection~\ref{smeth}.
\end{remark}

It is enough to prove Theorem~\ref{mthm} for \(1<p\leq2\).  Once the
\(\ell^2\) estimate is known, interpolation with the contractive
\(\ell^\infty\) bound gives the exponents \(2<p<\infty\).  The theorem concerns
all radii, rather than a dyadic or otherwise restricted family.   We do not consider the full
discrete spherical maximal function here.

The history begins with the continuous Hardy--Littlewood maximal problem.
Classical covering arguments give boundedness in every fixed dimension, but
their constants deteriorate as \(d\) grows.  Stein~\cite{Ste82}, followed by
Stein and Str\"omberg~\cite{SS83}, proved that Euclidean ball averages are
bounded on \(L^p(\R^d)\), for every \(p>1\), with a constant depending only on
\(p\).  The method of rotations uses the rotational symmetry of the Euclidean
ball. Bourgain~\cite{Bou86}
then obtained the dimension-free \(L^2\) estimate for every symmetric convex
body.  Bourgain~\cite{BouLp} and, independently,
Carbery~\cite{Car86} extended this conclusion to \(p>3/2\).  For the special
family of \(\ell^q\) balls, M\"uller~\cite{Mul90} reached the full range
\(p>1\).  These results connect dimension-free maximal estimates with convex
geometry and Fourier decay.  They also show that covering arguments alone do
not account for dimension-free bounds.

The passage from \(\R^d\) to \(\Z^d\) cannot be obtained by direct
transference.  Kosz,
Mirek, Plewa, and Wr\'obel~\cite[Theorem 1]{KMP23} proved that the optimal
continuous maximal constant is no larger than its discrete counterpart.  In
this precise sense, the lattice problem contains the continuous one.  In
addition, Bourgain, Mirek, Stein, and Wr\'obel~\cite[Theorem 2 and its proof
on p.~871]{BMS19} constructed symmetric ellipsoids for which the discrete
maximal norm is bounded below by a constant times
\((\log d)^{1/p}\).  The loss already occurs among radii smaller than
\(\sqrt2\).  Thus the general continuous theory for symmetric convex bodies
has no discrete analogue, and an anisotropic deformation may
destroy dimension-free control.  The equal coefficients in the Euclidean
quadratic form and the arithmetic of sums of squares therefore cannot be
treated as perturbative features of Stein's question.
For comparison, the same authors~\cite[Sections 1 and 3]{BMS19} proved full
dimension-free estimates for discrete cubes when \(p>3/2\), as well as their
dyadic counterpart for every \(p>1\).  The different behavior of cubes,
ellipsoids, and Euclidean balls shows that the discrete problem depends on the
geometry of the averaging sets.

The first dimension-free results treated separate ranges of the radius
parameter.  Bourgain, Mirek, Stein, and
Wr\'obel~\cite[Theorem 1.1]{BMS20} proved dimension-free estimates on
\(\ell^p(\Z^d)\), \(2\leq p\leq\infty\), along dyadic radii.  The same
authors~\cite[Theorem 2]{BMS21} treated every radius \(t\geq Cd\), where a
comparison with continuous averages becomes effective.  Kosz, Mirek, Plewa,
and Wr\'obel~\cite[Sections 1 and 2]{KMP23} developed comparison principles
for further convex bodies and identified boundary errors that obstruct their
use at smaller scales.  For spheres, Mirek, Szarek, and
Wr\'obel~\cite[Theorem 1.1]{MSW24} obtained the corresponding dyadic estimate
for \(2\leq p\leq\infty\) in dimensions at least five.  Dyadic control avoids
the accumulation of changes at consecutive squared radii, but it does not by
itself yield a full maximal theorem.

Two recent results are directly relevant to \eqref{e115}.  Mirek, Szarek, and
Wr\'obel~\cite[Theorem 1.1 and (1.12)]{MSW25}
proved dimension-free maximal, jump, variation, and oscillation estimates for
normalized discrete Gaussians throughout \(1<p<\infty\).  For nonnegative
functions, their Gaussian maximal function is pointwise dominated by the
Euclidean ball maximal function.  Consequently, their Gaussian estimate
provides the appropriate model bound, but the direction of the pointwise
domination does not imply the ball estimate.  Niksi\'nski and
Wr\'obel~\cite[Theorems 1.1 and
1.2]{NW26} subsequently established estimates over every radius on
\(\ell^p\), \(2\leq p\leq\infty\), for both balls and spheres in the range
\(0\leq t\leq d^{1/2-\varepsilon}\), for every \(\varepsilon>0\).  Their
proof combines uniform lattice-point estimates, a reduction to coordinate
layers, Fourier analysis, and a Rademacher--Menshov argument.  It was the
first dimension-free treatment of all small radii for these sets, rather than
only a lacunary subfamily.

Niksi\'nski~\cite[Theorems 1.2--1.5 and Proposition 7.1]{N26} has since
extended several of these ideas to convex bodies invariant under coordinate
permutations and sign changes.  By introducing a discrete isotropic parameter,
he obtained concentration estimates for lattice points, dimension-free bounds
over dyadic radii, and full estimates in an appropriate small-scale regime.
His frequency analysis also exhibits the neighborhoods of \(0\) and
\((1/2,\ldots,1/2)\) that are relevant in the discrete setting.  Nevertheless,
in the August 2026 version, Niksi\'nski~\cite[Question 1.8]{N26} still stated
the full Euclidean problem as open: these results do not control all
consecutive squared radii in the intermediate range.  Three contemporaneous
manuscripts subsequently obtained the full ball estimate: Jin and
Su~\cite[Theorem 1.1]{JS26}, Hormozi, Niksi\'nski, and
Wr\'obel~\cite[Theorem 1]{HNW26}, and the present paper.  The latter authors
also treat the full spherical maximal operator for $d\geq5$ and
$2\leq p\leq\infty$; see Hormozi, Niksi\'nski, and
Wr\'obel~\cite[Theorem 2]{HNW26}.  Their spherical result is dimension-free
in this range, whereas the optimal fixed-dimensional boundedness range remains
$p>d/(d-2)$.

Lattice-point asymptotics form a closely related part of the subject.
Mazo and Odlyzko~\cite[Theorems 1 and 2]{MO90} studied high-dimensional
spheres, while Niksi\'nski and Wr\'obel~\cite[Theorems 1.3 and 1.4]{NW26}
derived estimates uniform in the small-radius parameters needed for their
maximal theorem.  Dymowski~\cite[Theorem 1.1]{D26} extended such estimates to
high-dimensional \(\ell^q\) balls and spheres for integral \(q\geq2\).  In
the Euclidean case his range is \(n\leq cd\), and hence it does not enter the
middle regime considered below.   Those results concern averaging over the center, not a supremum over
radii.  They provide another application of Gaussian and theta function
methods to high-dimensional counting, while the maximal problem requires
uniform operator estimates after normalization at every radius.

Before the contemporaneous works of Jin and Su~\cite{JS26}, Hormozi,
Niksi\'nski, and Wr\'obel~\cite{HNW26}, and the present paper, combining the
small-radius theorem of Niksi\'nski and Wr\'obel~\cite[Remark 1]{NW26} with
the large-radius comparison left, up to fixed constants, the interval
\[
 d^{1/2-\varepsilon}<t<Cd.
\]
Coordinate sparsity is no longer available there, the relative boundary error
in the continuous comparison remains too large, and the averaging set may
change at every integer value of \(t^2\).  Both proofs control the full
sequence of squared radii in this interval and connect the small-radius theory
with the large-radius comparison theorem.  In the present argument, exact
theta coefficients are combined with a single saddle analysis based on
\(L=\min(n,d)\).  Parity produces a second major arc, whose contribution is
retained as part of the principal expansion.

\subsection{Methods and outline}\label{smeth}

Earlier arguments exploit models adapted to a particular scale.  Bourgain,
Mirek, Stein, and Wr\'obel~\cite[Sections 3--5]{BMS20} used dimension-free
Fourier multiplier estimates along dyadic radii, whereas their large-radius
argument~\cite[Section 5]{BMS21} compares lattice averages with continuous
ones after the radius reaches a constant multiple of \(d\).  Niksi\'nski and
Wr\'obel~\cite[Sections 2--5]{NW26} reduced the small-radius problem to
coordinate layers and supported that reduction with uniform lattice-point
estimates.  Niksi\'nski~\cite[Sections 6--8]{N26} combined concentration,
dimension reduction, diffusion semigroups, and analysis near the two relevant
frequency points for general \(1\)-symmetric bodies.  None of these
scale-specific reductions remains uniform throughout the intermediate
interval.

The normalized Gaussian theorem supplies the model family used in our final
operator estimate.  Although these Gaussian averages do not form a semigroup,
Mirek, Szarek, and Wr\'obel~\cite[Sections 3--7]{MSW25} controlled them through
two theta function representations, fractional derivatives, complex
interpolation, and comparison with the discrete heat semigroup.  We use their
maximal theorem as an input.  Our argument compares each exact ball
multiplier, at every consecutive squared radius, with finitely many Gaussian
terms and produces a remainder that is summable uniformly in the dimension.

Jin and Su~\cite[Sections 3--7]{JS26} also begin with the exact theta
coefficient for the ball and retain the contributions of the arcs centered
at \(0\) and \(\pi\).  They choose the saddle of the theta power
\(h(z)^d\), treat the factor \((1-z)^{-1}\) as an amplitude, and divide the
finite radii into small, critical, and growing regimes.  In the growing
regime, they separate the remote arcs by Abel summation, a quadratic Weyl sum
estimate, and quadratic Gauss sums.  Their higher order terms are radial
derivatives of normalized Gaussian multipliers; finite differences at nearby
parameters replace those derivatives by genuine Gaussian multipliers.

Our treatment differs from that of Jin and Su in the following respects.  We choose the stationary point of the
complete ball generating function \(h(z)^d/(1-z)\).  This choice cancels the
linear term in the full logarithmic phase and leads to one set of scale
parameters throughout \(1\leq n\leq Ad^2\).  An elementary rigidity argument
for quadratic phases identifies the two possible major arcs and gives a
uniform gap on their complement; in particular, our proof does not use Weyl
sum estimates or Gauss sums.  On the two arcs, we keep the normalized moments
of the phase integral exact and apply Taylor's formula only to the Fourier
amplitude.  Derivatives of the phase beyond order four are therefore not
needed for the coefficient normalization.  Finally, a small shift of the Gaussian parameter allows us to dominate
each polynomially weighted Gaussian kernel by two normalized discrete
Gaussian kernels with nearby parameters.  These three choices provide a single
multiplier expansion of arbitrary prescribed order, with a remainder that
becomes summable after interpolation.

The ball argument of Hormozi, Niksi\'nski, and Wr\'obel
\cite[Proposition 4 and Sections 2.3--2.7]{HNW26} is closer to the present
one.  After identifying their parameter $r$ with our $\rho$, their
saddle-point equation is precisely \eqref{e008}.  Both arguments use the
complete generating function $h(z)^d/(1-z)$, retain the two arcs centered at
$0$ and $\pi$, control complex theta quotients uniformly in frequency, and
expand the corresponding multipliers to arbitrary order.  Hormozi,
Niksi\'nski, and Wr\'obel replace derivatives in the Gaussian parameter by
high order finite differences, thereby obtaining finite linear combinations
of normalized discrete Gaussian multipliers and their half-lattice
translates, with a remainder of arbitrary polynomial order.  Their method
also applies to discrete spheres.  In the present proof, Proposition~\ref{invp}
provides a deterministic quadratic-phase argument for the global two-arc
localization, while Lemma~\ref{jker} controls the local Taylor kernels by
pointwise comparison with normalized Gaussians at nearby parameters.  These
are differences in implementation; the complete saddle-point and Gaussian
comparison framework is shared by the two papers.

We now briefly outline our approach. First, we represent each ball multiplier by a coefficient in one complex
variable.
Since \(tB\) changes only when \(t^2\)
crosses an integer, we write
\begin{equation*}
 \sqrt nB:=\{x\in\Z^d:\abs{x}^2\leq n\},\qquad
 n\in\N_0:=\{0,1,2,\ldots\},
\end{equation*}
and let  
\begin{equation*}
  \mathfrak m_{\sqrt{n}}(\xi):=  \frac{1}{\#(\sqrt{n}B)}
     \sum_{y\in \sqrt{n}B} \exp(2\pi i y\cdot \xi),\quad \xi\in \mathbb{T}^d,
\end{equation*}
be the Fourier multiplier of \(M_{\sqrt n}\).  We use the Jacobi theta
functions
\begin{equation}\label{e006}
\vartheta(z,u):=\sum_{k\in\Z}z^{k^2}\exp{(2\pi i ku)}, \qquad h(z):= \vartheta(z,0) = \sum_{k\in\Z}z^{k^2},
\end{equation}
where \(z\in\mathbb C\), \(\abs z<1\), and \(u\in\mathbb T\).

The coefficient in the numerator is the Fourier-weighted lattice-point count,
whereas the denominator equals \(\#(\sqrt nB)\):
\begin{equation}\label{e007}
 \mathfrak m_{\sqrt n}(\xi)=
 \frac{[z^n]\left((1-z)^{-1}\prod_{j=1}^d\vartheta(z,\xi_j)\right)}
 {[z^n]\left((1-z)^{-1}h(z)^d\right)}.
\end{equation}
See Section  \ref{spre} for the proof. Given a differentiable function of one complex variable, we write
\((Df)(z):=zf'(z)\) for the Euler differential operator.  In particular,
\((Df)(\rho)=\rho f'(\rho)\) on the real interval \((0,1)\).  For
\(n\geq1\), we define \(\rho:=\rho_{n,d}\in(0,1)\) as the unique solution of
the saddle-point equation
\begin{equation}\label{e008}
 n=dD\log h(\rho)+\frac{\rho}{1-\rho}.
\end{equation}
The saddle equation \eqref{e008} differs from the equation
\(n=dD\log h(\rho)\) used by Niksi\'nski and
Wr\'obel~\cite[Theorem 3.4 and Lemma 3.5]{NW26} for the spherical generating
function \(h(z)^d\).  Jin and Su~\cite[Subsection 3.2]{JS26} use the latter
equation for the theta power in the ball coefficient and keep
\((1-z)^{-1}\) in the amplitude.  Both choices are valid.  We instead take
the logarithmic derivative of the complete ball generating function
\(h(z)^d/(1-z)\), which gives \eqref{e008}.
Consequently, the derivative of the full phase vanishes at the origin.  This
exact cancellation is useful when \(\rho\) approaches \(1\), because the
contribution \(\rho/(1-\rho)\) is incorporated into the saddle rather than
estimated separately.

Second, we prove a multiplier expansion of any prescribed fixed order,
uniformly on the frequency torus.  For each fixed \(A>1\), we put
\(L:=\min(n,d)\) and consider \(n\leq Ad^2\).  Cauchy's formula converts
\eqref{e007} into an angular integral.  Parity produces two candidate saddle
points, at \(0\) and \(\pi\), and the saddle equation \eqref{e008} cancels the
linear term at the origin.  The deterministic rigidity statement in
Proposition~\ref{invp} shows that a quadratic theta phase can be almost
constant only near one of these two points.  It follows that the remaining
part of the contour has an exponential gap after the \(d\)-fold product is
taken.  

Third, we identify the inverse Fourier kernels of the Taylor terms.
They are polynomials in the centered statistic
\(\abs{x}^2-dm(\rho)\), multiplied by a normalized discrete Gaussian.  By slightly shifting the Gaussian parameter, we dominate each of these
kernels pointwise by the sum of two normalized discrete Gaussian kernels.  The theorem of Mirek, Szarek, and
Wr\'obel~\cite[Theorem 1.1, in particular (1.5)]{MSW25} then controls the
finite Taylor sum.  The remainder is uniformly bounded on \(\ell^1\) and has
\(\ell^2\) norm
\(O(L^{-(N+1)/2})\).  Using the interpolation then completes the proof for every \(1<p<2\) once \((N+1)(p-1)\geq2\).  The
large-radius theorem of Bourgain, Mirek, Stein, and Wr\'obel gives the estimate
for the remaining radii.

This paper is organized as follows. In Section~\ref{spre} we establish the coefficient representation and uniform
estimates for the saddle parameters, localize the coefficient integral to
major arcs centered at \(0\) and \(\pi\), and prove the weighted coefficient
expansion.  In Section~\ref{sprf} we determine the contribution of the
\(\pi\)-centered major arc, prove the multiplier expansion of arbitrary
fixed order, estimate the Taylor kernels by nearby Gaussians, and combine
these estimates with the Gaussian and large-radius maximal theorems.

\subsection{Notation}\label{snot}

We will use similar notation  as in \cite{NW26}. Put
\(\N:=\{1,2,\ldots\}\) and \(\N_0:=\N\cup\{0\}\). The letter \(d\in\N\)
always stands for the dimension, while \(n\in\N_0\) denotes a squared
radius.   For a finite set \(A\),
we denote its cardinality by \(\#A\).

 We abbreviate
\(\es{s}:=\exp(-2\pi i s)\), identify \(\T^d:=\R^d/\Z^d\) with the unit cube
\([-1/2,1/2)^d\), and equip \(\T^d\) with the normalized Haar measure.
For \(f\in\ell^1(\Z^d)\), its Fourier series is
\begin{equation}\label{e012}
 \widehat f(\xi):=\sum_{x\in\Z^d}f(x)\,\es{x\cdot\xi},  \qquad \xi\in\T^d.
\end{equation}
If \(m\in L^1(\T^d)\), we use the inverse transform
\begin{equation}\label{e163}
 \mathcal F^{-1}m(x):=\int_{\T^d}m(\xi)\es{-x\cdot\xi}\dd\xi,
 \qquad x\in\Z^d.
\end{equation}
For \(f,g\in\ell^1(\Z^d)\), discrete convolution is defined by
\begin{equation}\label{e164}
 (f*g)(x):=\sum_{y\in\Z^d}f(x-y)g(y), \qquad x\in\mathbb{Z}^d.
\end{equation}
With the conventions \eqref{e012}, \eqref{e163}, and \eqref{e164},
\(\widehat{f*g}=\widehat f\,\widehat g\), and a
multiplier \(m\) acts by convolution with \(\mathcal F^{-1}m\).
For \(\xi\in\T^d\), the expression \(\xi+1/2\) means
\((\xi_1+1/2,\ldots,\xi_d+1/2)\) modulo \(\Z^d\).  Distances on
\(\T\) are taken modulo \(\Z\); in particular,
\(\norm{v}_{\T}:=\dist(v,\Z)\) for $v\in\R$.  The symbol \([z^n](H) \) denotes the
coefficient of \(z^n\) in a power series \(H\); for instance, if $H(z)=\sum c_n z^n$, then $[z^n](H)=c_n$. 

For nonnegative quantities \(X\) and \(Y\), we write
\(X\lesssim_A Y\) if \(X\leq C_AY\), and \(X\approx_A Y\) if both
\(X\lesssim_A Y\) and \(Y\lesssim_A X\); the subscript is omitted when
the constants are universal.  The same convention applies to big-O
notation, and all implicit constants are independent of \(d\).  Auxiliary
absolute parameters are fixed in the order stated in the proofs.  Finally,
\(\norm{f}_p\) abbreviates \(\norm{f}_{\ell^p(\Z^d)}\) whenever no
confusion can arise.  

\section{Preliminary Lemmas}\label{spre}

In this section we establish the coefficient estimates used in the multiplier
expansion.  In Lemma~\ref{spar} we obtain a common scale for the quadratic saddle
term and every fixed logarithmic derivative.  In Lemma~\ref{trat} and
Proposition~\ref{twow}
we localize the coefficient integral to the major arcs centered at \(0\) and
\(\pi\).  In Proposition~\ref{hwt} we then prove a finite Taylor expansion of any
prescribed order, and in Lemma~\ref{jker} we identify the corresponding
convolution kernels.  This order of presentation isolates the two analytic
tasks---localizing the coefficient integral and controlling the resulting
kernels---before they are assembled in Section~\ref{sprf}.

We first justify the coefficient formula used throughout the paper.  For
\(\abs z<1\), absolute convergence permits us to expand and rearrange the
series:
\[
 \begin{aligned}
 \frac{1}{1-z}\prod_{k=1}^d\vartheta(z,\xi_k)
 =\sum_{j\geq0}\sum_{x\in\Z^d}
 z^{j+\abs{x}^2}\es{x\cdot\xi}
 =\sum_{n\geq0}\left(
 \sum_{\substack{x\in\Z^d\\\abs{x}^2\leq n}}
 \es{x\cdot\xi}\right)z^n.
 \end{aligned}
\]
Thus the coefficient of \(z^n\) is the Fourier transform of the indicator
of \(\sqrt nB\).  Letting \(\xi=0\) gives \(\#(\sqrt nB)\).  Their ratio is
precisely the multiplier of \(M_{\sqrt n}\), proving \eqref{e007} directly.

Recall that \((Df)(\rho)=\rho f'(\rho)\).  We put
\begin{equation}\label{e016}
 m(\rho):=D\log h(\rho)=\frac1{h(\rho)}\sum_{k\in\Z} k^2 \rho^{k^2},\qquad
 a(\rho):=\frac{\rho}{1-\rho},\qquad
 \tau_\rho:=\frac{1}{1-\rho}.
\end{equation}
Through a direct differentiation we obtain
\begin{equation}\label{e131}
 Dm(\rho)=D^2\log h(\rho)=\frac1{2 h(\rho)^2} \sum_{j,k\in\Z}(j^2-k^2)^2
  \rho^{j^2+k^2}, \qquad
 Da(\rho)=\frac{\rho}{(1-\rho)^2}.
\end{equation}
So both the above two quantities are strictly positive. As a result, we see that the function 
\(\rho \mapsto d\,m(\rho)+a(\rho)\) increases strictly on $(0,1)$   with range $(0,\infty)$, and
\eqref{e008} has a unique solution for every \(n\geq1\).

We will use a division between two parameter ranges repeatedly.   For a large absolute constant \(s_*\), we put
\(\rho_*:=\exp(-\pi/s_*)\), and call \(\rho\leq\rho_*\) the compact
range, \(\rho>\rho_*\) the dense
range, respectively.  The value of \(s_*\) (and then $\rho_*$) may be increased relying on the specific context, which
changes only the absolute constants.

Recall Jacobi's triple product formula
(see e.g. \cite[(20.5.9)]{RW10})
\begin{equation}\label{e019}
 \vartheta(z,u)=\prod_{j\geq1}(1-z^{2j})
 (1+z^{2j-1}\e^{2\pi iu})(1+z^{2j-1}\e^{-2\pi iu}),
 \qquad \abs z<1,\ u\in\R.
\end{equation}
In particular, one has
\begin{equation}\label{e020}
 h(\rho)=\prod_{j\geq1}(1-\rho^{2j})
 (1+\rho^{2j-1})^2, \qquad 0<\rho<1.
\end{equation}
Given $S\in\{z\in\C: \Rea z>0 \}$, we define
\[
 H_S(u):=\sum_{j\in\Z}\e^{-\pi S(j-u)^2}, \qquad u\in\R.
\]
The following Landen's transformations (see e.g. \cite[(20.7.32)]{RW10}) will be useful:
\begin{equation}\label{e014}
 \vartheta(\e^{-\pi/S},u)=S^{1/2}H_S(u),
 \qquad \Rea S>0,\  u\in\R,
\end{equation}
where \(S^{1/2}\) denotes the principal square root.

\begin{lemma}\label{spar}
For \(0<\rho<1\) and every fixed integer \(r\geq2\),
\begin{equation}\label{e025}
 m(\rho)\approx a(\rho),\qquad
 D^2\log h(\rho)\approx   \rho\,\tau_\rho^2,
 \qquad
 \abs{D^r\log h(\rho)}\lesssim_r
   \rho\,\tau_\rho^r.
\end{equation}
If \(n,d\in\N\), \(\rho:=\rho_{n,d}\), and \(L:=\min(n,d)\ge1\).  Then we have
\begin{equation}\label{e017}
 n\approx d\,a(\rho), \qquad
 L\approx d\rho, \qquad
 dD^2\log h(\rho)\approx L\tau_\rho^2,
\end{equation}
and
\begin{equation}\label{e018}
 V :=dD^2\log h(\rho)+\frac{\rho}{(1-\rho)^2} \approx L\tau_\rho^2.
\end{equation}
The implicit constants in \eqref{e025} are uniform in \(\rho\), and all constants
in  \eqref{e017}-\eqref{e018} are uniform in \(n\) and \(d\).
\end{lemma}

\begin{proof}
Consider initially that \(0<\rho\leq\rho_*\). Via  \eqref{e016} and \eqref{e131}, we immediately get the first two estimates in  \eqref{e025}.  Note that the logarithmic
series obtained from the Jacobi product formula \eqref{e020} converges absolutely and uniformly after any
prescribed times of applications of \(D\). To obtain the bound for higher derivatives in \eqref{e025}, we write
\(\log h(\rho)=\sum_{j\geq1}b_j\rho^j\).  Absolute convergence   implies
\(\sum_{j\geq1}j^r\abs{b_j}\rho_*^j<\infty\) for each fixed
\(r\).
Consequently,
\[
 \abs{D^r\log h(\rho)}
 \leq \rho\sum_{j\geq1}j^r\abs{b_j}\rho_*^{j-1}
 \lesssim_r\rho.
\]
The implicit constants of course depend only on the fixed threshold $\rho_*$.

For \(\rho_*<\rho<1\), we put \(\rho=\e^{-\pi/s}\), so that \begin{equation} \label{sran} 
    s>s_*:=\frac{\pi}{-\log \rho_*}.
\end{equation}
An application of  \eqref{e014} yields
\begin{equation}\label{e023}
 h(\e^{-\pi/s})=s^{1/2}h(\e^{-\pi s}).
\end{equation}
Since
\[
 \log h(\e^{-\pi/s})=\tfrac12\log s+\log h(\e^{-\pi s}),
 \qquad D=\frac{s^2}{\pi}\frac{\dd}{\dd s},
\]
and the defining series of $h(\e^{-\pi s}) $ and all its derivatives converge absolutely and
uniformly for \(s\geq s_*\).  It follows from  repeated differentiations with some routine estimates that, for any
\(r\geq0\),
\[
 D^r h(\e^{-\pi s})
 =
 O_r\bigl(s^{2r}\e^{-\pi s}\bigr).
\]
and hence 
\begin{equation*}
 D^r\log h(\rho)
 =\frac{(r-1)!}{2\pi^r}s^r+O_r(s^{2r}\e^{-\pi s}).
\end{equation*}
After increasing \(s_*\), this yields
\begin{equation*}
 m(\rho)=\frac{s}{2\pi}+O(s^2\e^{-\pi s}),\qquad
 D^2\log h(\rho)\approx s^2,\qquad
 \abs{D^r\log h(\rho)}\lesssim_rs^r.
\end{equation*}
On the other hand,
\(a(\rho)=(\e^{\pi/s}-1)^{-1}\approx s\), while
\(\tau_\rho=(1-\e^{-\pi/s})^{-1}\approx s\).  Then \eqref{e025}
also holds in the dense range.

The saddle-point equation \(n=d\,m(\rho)+a(\rho)\) in \eqref{e008} and
\eqref{e025} yield
\(n\approx d\,a(\rho)\).  It follows that
\[
 L=\min(n,d)\approx d\min(a(\rho),1)\approx d\rho.
\]
Then the second estimate in \eqref{e025}
now gives
\[
 dD^2\log h(\rho)\approx d\rho\tau_\rho^2
 \approx L\tau_\rho^2.
\]
The term \(\rho/(1-\rho)^2=\rho\tau_\rho^2\) in the definition of \(V\) in
\eqref{e018} is no larger
than \(d\rho\tau_\rho^2\) because \(d\geq1\),  whence it is also
\(O(L\tau_\rho^2)\).
This proves \eqref{e017} and \eqref{e018}.
\end{proof}

\subsection{Local arc estimates}\label{sthe}

Let $\rho\in(0,1)$. We will analyze the ratios of two theta function  in a complex neighborhood of the circle
\(\abs{z}=\rho\).  Define
\begin{equation}\label{e029}
 r_\rho(\theta,u):=\frac{\vartheta(\rho\e^{i\theta},u)}
 {\vartheta(\rho\e^{i\theta},0)},\qquad
 \ell_\rho(u):=-\log r_\rho(0,u), \qquad  u\in\T.
\end{equation}
Since \(0<r_\rho(0,u)\le1\) by Jacobi's triple product \eqref{e019}, the quantity \(\ell_\rho(u)\) is nonnegative. The logarithm in  \eqref{e029} is initially defined at $\theta=0$. We
now  extend its definition to the whole $\R/(2\pi\Z)$. In fact, write
\begin{equation}\label{e034}
 \mathcal L(z,u):=
 2\sum_{k\geq1}(-1)^{k+1}
 \frac{\cos(2\pi ku)-1}{k}
 \frac{z^k}{1-z^{2k}},
 \qquad \abs{z}<1,\quad u\in\T.
\end{equation}
Then the series converges locally uniformly in the  unit disk and hence defines
an analytic function of \(z\).  We claim that
\begin{equation} \label{ff01} 
     \exp\mathcal L(z,u)
     =
     \frac{\vartheta(z,u)}{\vartheta(z,0)}.
\end{equation}
Once \eqref{ff01} is established, we see  \eqref{e034} actually determines an analytic logarithmic branch;  and moreover, we can define the logarithm in
\eqref{e029} globally by letting \(z:=\rho\e^{i\theta}\) and 
\begin{equation} \label{logd} 
    \log r_\rho(\theta,u):=
     \mathcal L(\rho\e^{i\theta},u), \qquad \theta\in\R/(2\pi\Z),\ u\in\T,
\end{equation}
which
agrees with
$-\ell_\rho(u)$ at $\theta=0$. This is enough for our purpose to derive local arc estimates subsequently. 

It remains to show \eqref{ff01}.  For $J\geq1$,
we put
\[
 P_J(z,u):=
 \prod_{j=1}^J
 \frac{(1+z^{2j-1}\e^{2\pi iu})
       (1+z^{2j-1}\e^{-2\pi iu})}
      {(1+z^{2j-1})^2}
\]
and
\[
 \begin{aligned}
 \mathcal L_J(z,u):=\sum_{j=1}^J\bigl(&
 \log(1+z^{2j-1}\e^{2\pi iu})
 +\log(1+z^{2j-1}\e^{-2\pi iu})\\
 &-2\log(1+z^{2j-1})\bigr),
 \end{aligned}
\]
where every logarithm can be explained by the following expansion
\[
 \log(1+w) =\sum_{k\geq1}\frac{(-1)^{k+1}}{k}w^k,
 \qquad \abs w<1.
\]
Then 
\[
 \begin{aligned}
 \lim_{J\to\infty}\mathcal L_J(z,u)
 &=\sum_{j,k\geq1}\frac{(-1)^{k+1}}{k}z^{(2j-1)k}
 \bigl(\e^{2\pi iku}+\e^{-2\pi iku}-2\bigr)\\
 &=2\sum_{k\geq1}(-1)^{k+1}
 \frac{\cos(2\pi ku)-1}{k}
 \sum_{j\geq1}z^{(2j-1)k}\\
 &=2\sum_{k\geq1}(-1)^{k+1}
 \frac{\cos(2\pi ku)-1}{k}
 \frac{z^k}{1-z^{2k}}
 =\mathcal L(z,u),
 \end{aligned}
\]
where we have used the  fact that the doubles series converges absolutely whenever 
$\abs z<1$, uniformly in $u\in\T$.
On the other hand, Jacobi's triple product formula \eqref{e019} gives
\[
 \lim_{J\to\infty}P_J(z,u)
 =\frac{\vartheta(z,u)}{\vartheta(z,0)}.
\]
Note that for each finite $J$, $
 \exp\mathcal L_J(z,u)=P_J(z,u).$ Passing to the limit in
$\exp\mathcal L_J(z,u)=P_J(z,u)$ proves \eqref{ff01}. 

\begin{lemma}\label{trat}
There is a constant \(\delta\in(0,1)\) such that, whenever
\(0<\rho<1\), \(\abs{\theta}\leq\delta/\tau_\rho\), \(u\in\T\)  and
\(j\in\N\),
\begin{equation}\label{e030}
 \abs{\partial_\theta^j\log r_\rho(\theta,u)}
 \lesssim_j\tau_\rho^j\ell_\rho(u)
\end{equation}
and
\begin{equation}\label{e031}
 \abs{r_\rho(\theta,u)}\leq\e^{-\ell_\rho(u)/2}.
\end{equation}
Moreover,
\begin{equation}\label{e032}
 \ell_\rho(u)\approx a(\rho)\norm{u}_{\T}^2,
 \qquad \mathrm{for\ all}\ \ 0<\rho<1,\ u\in\T.
\end{equation}
\end{lemma}

\begin{proof}
We will treat the compact range and dense range separately like before. Recall \eqref{logd} and \eqref{e034}. Consider first that  \(\rho\leq\rho_*\). We first show  \eqref{e032} in this rage. By  \eqref{e006} and the fact $|\sin(\pi u)|\approx \norm u_\T $, we have  
\begin{equation*}
 1-r_\rho(0,u)=\frac{2}{h(\rho)}
 \sum_{k\geq1}\rho^{k^2}\big(1-\cos(2\pi ku)\big) \ge \frac{2}{h(\rho)} \big(1-\cos(2\pi u)\big) \gtrsim  \rho\norm u_\T^2.
\end{equation*}
Moreover,
\begin{equation*}
 1-\cos(2\pi  k u) = 2\sin^2(\pi  k u) \leq2 k^2\sin^2(\pi u),
\end{equation*}
where the inequality follows from the factorization
\(1- z^ k=(1-z)(1+z+\cdots+z^{ k-1})\) with
\(z=\e^{2\pi iu}\).
Then the trivial estimate \(\sum_{k\geq1}k^2\rho^{k^2}\lesssim\rho\) implies that $1-r_\rho(0,u) \lesssim \rho\norm u_\T^2$.  On the other hand, Jacobi's formula
\eqref{e019} gives, for \(0<\rho\leq\rho_*\),
\[
 r_\rho(0,u)
 =
 \prod_{j\geq1}
 \frac{\abs{1+\rho^{2j-1}\e^{2\pi iu}}^2}
      {(1+\rho^{2j-1})^2} \geq
       \prod_{j\geq1}
       \left(
       \frac{1-\rho_*^{2j-1}}
            {1+\rho_*^{2j-1}}
       \right)^2
       =:c_*>0,
\]
where we have used that
$
 \abs{1+\rho^{2j-1}\e^{2\pi iu}}
 \geq 1-\rho^{2j-1}.
$
Note that $c_*$ depends only on $\rho_*$. As a result, \(r_\rho(0,u)\) is bounded away
from zero uniformly, so 
$$-\log r_\rho(0,u)\approx1-r_\rho(0,u) \approx \rho \norm{u}_{\T}^2,$$ which proves \eqref{e032} on
the compact range, since $a(\rho)\approx \rho$. Next we show   \eqref{e030}. It holds uniformly for \(\abs z\leq\rho_*\) that
\[
 \left|\partial_\theta^j\frac{z^ k}{1-z^{2 k}}\right|
 \leq C_j k^j\rho^ k,\qquad j\geq0.
\]
Therefore,  by straightforward differentiations we obtain
\[
 \abs{\partial_\theta^j\log r_\rho(\theta,u)}
 \lesssim_j\sin^2(\pi u)\sum_{ k\geq1} k^{j+1}\rho^ k
 \lesssim_j  \ell_\rho(u),
\]
which proves for the compact range, since $\tau_\rho\approx1$. It remains to prove \eqref{e031} on the compact range.
By  \eqref{e030}, the integration along the real
line segment $\overline{0,\theta}$ from \(0\) to \(\theta\) yields
 \begin{align}
 \log\abs{r_\rho(\theta,u)}
 &=
 \Rea\log r_\rho(\theta,u) \nonumber\\
 &=
 -\ell_\rho(u)
 +\Rea\int_{\overline{0,\theta}}
 \partial_v\log r_\rho(v,u)\dd v  \label{inal} \\
 &\leq
 -\ell_\rho(u)
 +C\abs{\theta}\tau_\rho\ell_\rho(u). \nonumber
 \end{align}
For \(\abs{\theta}\leq\delta/\tau_\rho\), the last term is bounded by
\(C\delta\ell_\rho(u)\). After decreasing \(\delta\), if necessary,
so that \(C\delta\leq1/2\), we can acquire that $
 \log\abs{r_\rho(\theta,u)}
 \leq-\ell_\rho(u)/2,$
which proves \eqref{e031} for the compact range.

Turn to the case \(\rho>\rho_*\). We write again \(\rho=\e^{-\pi/s}\) with  \eqref{sran}, and define
\begin{equation}\label{e036}
 S(\theta):=\frac{s}{1-is\theta/\pi} = \frac{s+is^2\theta/\pi}{1+s^2\theta^2/\pi^2} .
\end{equation}
If \(\abs{\theta}\leq\delta/s\), then
\begin{equation}\label{e037}
 \Rea S(\theta)\approx\abs{S(\theta)}\approx s \approx \tau_\rho,\qquad
 S^{(m)}(\theta)=m!\left(\frac{i}{\pi}\right)^mS(\theta)^{m+1}.
\end{equation}
By \eqref{e014} we have
\begin{equation}\label{e038}
 \vartheta(\rho\e^{i\theta},u)=S(\theta)^{1/2}H_{S(\theta)}(u). 
\end{equation}
Consequently,
\begin{equation*}
 \log r_\rho(\theta,u)=F(S(\theta),u),\quad \mathrm{with} \quad
 F(S,u):=\log H_S(u)-\log H_S(0).
\end{equation*}

From periodicity and evenness, we may assume that
\(u=\norm{u}_{\T}\in[0,1/2]\).  If \(u\leq1/4\), we
factor out the main contribution 
point:
\begin{equation*}
 H_S(u)=\e^{-\pi Su^2}(1+E_S(u)),\qquad
 E_S(u):=\sum_{ j\neq0}
 \e^{-\pi S( j^2-2 ju)}.
\end{equation*}
For \(0\leq u\leq1/4\), it holds trivially that
\[
  j^2-2 j u\geq\tfrac12 j^2.
\]
Together with \(\Rea S\approx s\), this gives
\[
 \left|\partial_S^k\partial_u^m
 \e^{-\pi S( j^2-2 j u)}\right|
 \leq C_{k,m}s^m(1+\abs j)^{2k+m}\e^{-cs j^2}.
\]
Summing over \( j\neq0\), we obtain
\begin{equation}\label{e041}
 \abs{\partial_S^k\partial_u^m E_S(u)}
 \lesssim_{k,m}s^m\e^{-cs}.
\end{equation}
After enlarging \(s_*\), estimate \eqref{e041} gives
\(\abs{E_S(u)}\leq1/2\) uniformly.  Hence
\(1+E_S(u)\) does not vanish.  Then we can choose its logarithm by continuation from
the positive real point \(S=s\).
Denote
\begin{equation}\label{e013}
 Q_S(u):=\log(1+E_S(u))-\log(1+E_S(0)).
\end{equation}
Notice that
\(H_S(u)\) is even by the change of
variables \( j\mapsto- j\).  Therefore,
\(Q_S(u)=F(S,u)+\pi Su^2\), with \(Q_S\) defined in \eqref{e013}, is even.
In particular,
\(Q_S(0)=\partial_uQ_S(0)=0\), and the same identities hold after any
fixed number of \(S\)-derivatives.  As a consequence,
\[
 \partial_S^kQ_S(u)
 =\int_0^u(u-v)\partial_v^2\partial_S^kQ_S(v)\dd v.
\]
The chain rule for \(\log(1+E_S)\), the bound \(\abs{E_S}\leq1/2\), and
\eqref{e041} show that the integrand is \(O_k(s^2\e^{-cs})\).  Hence, for
every fixed \(k\geq0\),
\begin{equation}\label{e042}
 \abs{\partial_S^kQ_S(u)}
 \lesssim_k u^2s^2\e^{-cs}. 
\end{equation}
Since \(F(S,u)=-\pi S u^2+Q_S(u)\), then at $S=s$, estimate \eqref{e042} gives  \eqref{e032} for $\rho>\rho_*$ and $ u\le1/4$ by enlarging $s_*$, due to now $s\approx \tau_\rho\approx a(\rho)$. And moreover,  by \eqref{e042} we also have
\begin{equation}\label{e043}
 \abs{\partial_S^kF(S,u)}
 \lesssim_k u^2s^{1-k},
 \qquad k\geq0.
\end{equation}

If \(1/4\leq u\leq1/2\), both nearest integer points
contribute at the same order, so we need to include both terms.  That is, 
\begin{equation} \label{factH} 
 H_S(u)=\e^{-\pi S u^2}
 \big(1+\e^{-\pi Sy}+E_S^{(2)}(u)\big),
\end{equation}
where
\begin{equation*}
    y:=1-2 u, \qquad
     E_S^{(2)}(u) :=\sum_{j\in\Z\setminus\{0,1\}}
     \e^{-\pi S(j^2-2j u)}.
\end{equation*}
For every remaining \(j\), the quantity
\(j^2-2j u\) is bounded below by
\(c(1+j^2)\) uniformly. Then  similar termwise
differentiation  as before implies that
\begin{equation}\label{e044}
    \abs{\partial_S^kE_S^{(2)}(u)}\lesssim_k\e^{-cs}, \qquad k\ge0.
\end{equation}
The first two terms in \eqref{factH} cannot cancel. In fact, there is a universal $c_0$ such that
\begin{equation}\label{e045}
 \abs{1+\e^{-\pi Sy}}\geq c_0>0.
\end{equation}
To see this, note that $\Rea S>0$; if
\(sy\geq1\), the second term $\e^{-\pi Sy}$ has modulus at most \(\e^{-c}\), so we are done; if
\(sy\leq1\), then \(\Rea{S}\,y\lesssim1, \,\abs{\Ima S}\,y\lesssim\delta\), which will yield  \eqref{e045} by choosing
 \(\delta\) small enough.
Meanwhile, after another increase of \(s_*\), the error in \eqref{factH} is at most
\(c_0/2\) by  \eqref{e044}.  It follows that
\begin{equation} \label{ffpp} 
   1\gtrsim  \abs{1+\e^{-\pi Sy}+E_S^{(2)}(u)}\geq c_0/2.
\end{equation}
Thus we can write
\begin{equation}\label{e128}
 F(S,u)=-\pi S u^2
 +\log\big(1+\e^{-\pi Sy}+E_S^{(2)}(u)\big)-\log H_S(0).
\end{equation}
The term \(-\log H_S(0)\) in \eqref{e128} and all of its $S$-derivatives of
any $k$-th order are
\(O_k(\e^{-cs})\).  By  \eqref{e044} and  \eqref{ffpp} we have, for
\(k\geq1\),
\[
 \left|\partial_S^k  \log(1+\e^{-\pi Sy}+E_S^{(2)}(u)) \right|
 \leq C_ky^k\e^{-csy}\leq C_ks^{-k}.
\]
We conclude that
\begin{equation}\label{e046}
 \abs{\partial_S^kF(S,u)}\lesssim_ks^{1-k},
 \qquad k\geq0,
\end{equation}
where the case \(k=0\) follows directly from the expression for
\(F(S,u)\) in \eqref{e128}.  At \(S=s\), the leading term
\(\pi s u^2\) in \(-F(s,u)\) is at
least \(\pi s/16\), whereas both logarithmic terms are \(O(1)\).  This
 gives  \eqref{e032} for $\rho>\rho_*$ and $1/4\le u\le1/2$ by enlarging $s_*$.  
 
To sum up, in the dense range we have proved \eqref{e032}; and moreover, from  \eqref{e043} and \eqref{e046} it follows that, for all $u\in[0,1/2]$,
\begin{equation}\label{ff90}
 \abs{\partial_S^kF(S,u)}\lesssim_k u^2 s^{1-k},
 \qquad k\geq0,
\end{equation}

We now transfer the estimates in \(S\) back to the angular variable $\theta$.  For
every fixed \(j\geq1\), the Fa\`a di Bruno formula expresses
\(\partial_\theta^jF(S(\theta),u)\) as a finite sum of the following terms:
\[
 \partial_S^kF(S(\theta),u)
 \prod_{r=1}^j\big(S^{(r)}(\theta)\big)^{k_r},
 \quad \mathrm{where} \quad
 k=\sum_{r=1}^jk_r,
 \quad \sum_{r=1}^jrk_r=j.
\]
By \eqref{e037}, each such product has size \(O_j(s^{j+k})\).
Estimates  \eqref{ff90} therefore shows that each term has
size \(O_j( u^2s^{j+1})\), whence 
\begin{equation*}
 \abs{\partial_\theta^jF(S(\theta),u)}
 \lesssim_js^j\ell_\rho(u),\qquad j\geq 1.
\end{equation*}
This gives  \eqref{e030} for the dense range, since \(\tau_\rho\approx s\).
Finally, utilizing again  \eqref{inal} gives
\begin{equation*}
 \log\abs{r_\rho(\theta,u)}
 \leq-\ell_\rho(u)+C\abs{\theta}\tau_\rho\ell_\rho(u),
\end{equation*}
which yields \eqref{e031} for the dense rage
by decreasing \(\delta\) once more so that \(C\delta\leq1/2\). The proof of Lemma  \ref{trat} is then completed. 
\end{proof}

For the \(\pi\)-centered major arc, we use the rational factor
\begin{equation}\label{e049}
 \omega_\rho:=\frac{1-\rho}{1+\rho},\qquad
 R_\rho(\theta):=\frac{(1+\rho)(1-\rho\e^{i\theta})}
 {(1-\rho)(1+\rho\e^{i\theta})}.
\end{equation}

We now prove a uniform gap outside neighborhoods of \(0\) and \(\pi\) in
the dense range.  For
\(s>0\), define
\begin{equation}\label{e052}
 \mu_s(k):=\frac{\e^{-\pi k^2/s}}{Z_s},\qquad
 Z_s:=\sum_{k\in\Z}\e^{-\pi k^2/s},
\end{equation}
and
\begin{equation}\label{e053}
 \Phi_s(\theta,u):=\sum_{k\in\Z}\mu_s(k)
 \e^{i(\theta k^2+2\pi uk)},\qquad \theta\in\R,\quad u\in\T.
\end{equation}
For our purpose, we also need  the following lemma. 
\begin{lemma}\label{qrig}
There are absolute constants \(\eta_*\in(0,1)\) and \(C>0\) such that, for
every \(H\in\N\), \(\beta\in\R\), and \(k\in\{1,2\}\),
\begin{equation}\label{e130}
 \max_{1\leq h\leq H}\norm{\beta h^k}_{\T}\leq\eta
 \quad\Longrightarrow\quad
 \norm{\beta}_{\T}\leq\frac{C\eta}{H^k},
 \qquad0<\eta\leq\eta_*.
\end{equation}
\end{lemma}

\begin{proof}
Choose the representative \(\beta_0\in[-1/2,1/2]\).  The case \(h=1\)
implies \(\abs{\beta_0}\leq\eta\); moreover,
\(\norm{\beta h^k}_{\T}=\norm{\beta_0h^k}_{\T}\) for integral \(h\).  If
\(\abs{\beta_0}H^k\leq1/4\), then no reduction modulo one occurs at
\(h=H\), and the hypothesis gives
\(\abs{\beta_0}H^k\leq\eta\).  Otherwise, take the least integer
\(h\leq H\) for which \(\abs{\beta_0}h^k\geq1/4\).  Necessarily
\(h\geq2\).  If \(k=1\), then the
minimality shows that
\[
 \abs{\beta_0}h<\frac14+\abs{\beta_0}
 \leq\frac14+\eta<\frac12;
\]
and if \(k=2\), minimality will yield
\(\abs{\beta_0}(h-1)^2<1/4\), and so
\[
 \abs{\beta_0}h^2
 <\frac14+\sqrt{\abs{\beta_0}}+\abs{\beta_0}<\frac12,
\]
when \(\eta_*\) is sufficiently small.  Thus, in either case,
\(1/4\leq\abs{\beta_0}h^k<1/2\), so
\(\norm{\beta h^k}_{\T}\geq1/4\), contrary to the hypothesis.  The first
alternative must hold, and \eqref{e130} follows.
\end{proof}

\begin{proposition}\label{invp}
After increasing the fixed \(s_*\), if necessary, there are absolute
constants \(\delta_*>0\) and \(C>0\) such that, for every
\(0<\delta_0\leq\delta_*\), there is
\(\varepsilon_0=\varepsilon_0(\delta_0)\in(0,1/2)\) for which the following
implication holds.  If
\(s\geq s_*\), \(\theta\in\R\), \(u\in\T\), and
\begin{equation}\label{e054}
 \abs{\Phi_s(\theta,u)}\geq1-\varepsilon_0,
\end{equation}
then, for some \(\nu\in\{0,1\}\) it holds that
\begin{equation}\label{e055}
 \dist(\theta,\nu\pi+2\pi\Z)\leq\frac{\delta_0}{s},\qquad
 \dist(u,\nu/2+\Z)\leq\frac{C\delta_0}{\sqrt{s}}.
\end{equation}
In particular, if \(\dist(\theta,\pi\Z)>\delta_0/s\), then
\begin{equation}\label{e056}
 \sup_{u\in\T}\abs{\Phi_s(\theta,u)}\leq1-\varepsilon_0.
\end{equation}
\end{proposition}

\begin{proof}
Notice that a comparison of the defining series for \(Z_s\) with its corresponding Gaussian integral
yields \(Z_s\approx\sqrt{s}\).  We take
\begin{equation*}
 H:=\lfloor\sqrt{s}\rfloor,\qquad
 I_H:=\{k\in\Z:\abs{k}\leq4H\}.
\end{equation*}
Once \(s_*\) is large enough, we must have
\begin{equation}\label{e057}
 H\geq\frac12\sqrt{s},\qquad
 \mu_s(k)\geq\frac{c_0}{\sqrt{s}}\quad(k\in I_H).
\end{equation}
For later simplicity we write
\begin{equation}\label{e051}
 P(k):=\theta k^2+2\pi uk,\qquad z_k:=\e^{iP(k)},\qquad
 M:=\sum_k\mu_s(k)z_k,\qquad \zeta:=\frac{M}{\abs M}.
\end{equation}
We may assume \(\varepsilon_0<1/2\), so \(M\neq0\) and \(\zeta\) in
\eqref{e051} is well defined.  Then the assumption
\eqref{e054} implies
\begin{equation}\label{e058}
 \sum_k\mu_s(k)\abs{z_k-\zeta}^2=2(1-\abs{M})
 \leq2\varepsilon_0.
\end{equation}
For a small \(\eta>0\), we define the sets of bad points and good points  by
\begin{equation}\label{e062}
 B_H:=\{k\in I_H:\abs{z_k-\zeta}>\eta\},\qquad
 G_H:=I_H\setminus B_H,
\end{equation}
respectively. From \eqref{e057}, \eqref{e058}, and \eqref{e062} it follows that
\begin{equation*}
 2\varepsilon_0
 \geq\sum_{k\in B_H}\mu_s(k)\abs{z_k-\zeta}^2
 \geq\frac{c_0\eta^2}{\sqrt{s}}\#B_H,
 \quad \mathrm{and\ so}\quad
 \#B_H\leq\frac{2\varepsilon_0\sqrt{s}}{c_0\eta^2}.
\end{equation*}
We thus can choose \(\varepsilon_0\leq c_1\eta^2\), with \(c_1\) small enough so
that \(\#B_H<H/10\).

We first locate the quadratic coefficient $\theta$.  For each \(1\leq h\leq H\),
since the number of points $k\in (B_H-h)\cup B_H\cup(B_H+h)$ with \(\abs{k}\leq2H\) is only at most
\[
 \#\big((B_H-h)\cup B_H\cup(B_H+h)\big)
 \leq3\#B_H<\frac{3H}{10},
\]
we may choose $k=k(h)$
for which \(k-h,k,k+h\in G_H\). 
This implies their corresponding points $z_{k-h},z_{k},z_{k+h} $  lie within the circle centered at \(\zeta\) of radius \(\eta\).
As a result,
\begin{equation*}
\abs{\e^{2i\theta h^2}-1} =  \left|z_{k+h}z_{k-h}\overline{z_k}^{\,2}-1\right| \lesssim\eta,
 \qquad \forall\, 1\leq h\leq H.
\end{equation*}
From this, 
with the notation \(\alpha:=\theta/\pi\), we obtain
\begin{equation}\label{e061}
 \norm{\alpha h^2}_{\T}\lesssim\eta,\qquad1\leq h\leq H,
\end{equation}
where we have used
\begin{equation}\label{e083}
 \abs{\e^{2\pi iv}-1}\approx\norm v_{\T},\qquad v\in\R.
\end{equation}
Consequently, Lemma~\ref{qrig}, applied with \(k=2\) to \eqref{e061}, and the relation
\(H^2\approx s\)  give
\begin{equation}\label{e063}
 \dist(\theta,\pi\Z)= \pi\norm \alpha_\T      \lesssim\frac{\eta}{s}.
\end{equation}

We next locate the linear coefficient $u$, by an  analogous method. Owing to  \eqref{e063},  we may write, up to modulo \(2\pi\),
\begin{equation}\label{e084}
 \theta=\nu\pi+\theta_0,\qquad \nu\in\{0,1\}\quad  \textrm{with} \quad
 \abs{\theta_0}\lesssim\frac{\eta}{s}.
\end{equation}
For each \(1\le h\le H\), we take the good triple \(k-h,k,k+h\)  chosen in the  preceding procedure. Then
\begin{equation}\label{e064}
 \abs{\e^{i(4\theta kh+4\pi uh)}-1} = | z_{k+h}\overline{z_{k-h}} -1 | \lesssim\eta.
\end{equation}
Since \(\abs{k}\leq2H\) and \(H^2\lesssim s\), the
bound for \(\theta_0\) in \eqref{e084} shows that 
\(|4\theta_0kh|\lesssim\eta\), while
\(4\nu\pi kh\in2\pi\Z\).  Therefore, \eqref{e064} and \eqref{e083} lead to that
\begin{equation*}
 \norm{2uh}_{\T}\lesssim\eta,\qquad1\leq h\leq H.
\end{equation*}
Hence by Lemma~\ref{qrig}, with \(k=1\) and \(\beta=2u\), we arrive at that
\begin{equation}\label{e066}
 \dist(2u,\Z)\lesssim\frac{\eta}{H}.
\end{equation}

It remains to select the parity.  To obtain the required adjacent
pair, let us inspect the disjoint pairs
\(\{2j,2j+1\}\subset[-2H,2H]\).  If each pair contained a bad point,
then we must have \(\#B_H\geq H\), which will be contrary to the fact that \(\#B_H<H/10\).
For this reason, there are two adjacent points \(k,k+1\), with \(\abs{k}\leq2H\), are good, which implies that
\begin{equation*}
 |\e^{i(\theta(2k+1)+2\pi u)}-1| = |z_{k+1}\overline{z_k}-1|\lesssim\eta.
\end{equation*}
Since \(\abs{k}\leq2H\), by  \eqref{e063} we have
\(|\theta_0(2k+1)|\lesssim \eta/\sqrt{s}\). Then we deduce that
\begin{equation}\label{e125}
 \dist(u+\nu/2,\Z)\lesssim\eta.
\end{equation}
Choose \(m\in\Z\) such that
\(\abs{2u-m}\lesssim\eta/H\), as permitted by \eqref{e066}.  If
\(m\not\equiv\nu\pmod2\), then \eqref{e125} has left side
\(1/2+O(\eta/H)\), which is impossible for small \(\eta\).  Thus
\(m\equiv\nu\pmod2\), and hence
\[
 \dist(u,\nu/2+\Z)\lesssim\frac{\eta}{H}.
\]

To complete the proof, we now specify the order of the parameter choices.  First enlarge \(s_*\)
so that \eqref{e057} holds.  Given
\(0<\delta_0\leq\delta_*\), take \(\eta:=c\delta_0\), where \(c\) is small
enough for both applications of Lemma~\ref{qrig} and for the constant in
\eqref{e063}.  We then set \(\varepsilon_0:=c_1\eta^2\).  Letting
\(\delta_*\) sufficiently small ensures all stated upper bounds on \(\eta\)
and \(\varepsilon_0\).  These choices prove \eqref{e055}, while its contrapositive gives
\eqref{e056}.
\end{proof}

Let $0<\rho<1$. We define the normalized theta quotient by setting:
\begin{equation}\label{e068}
 \Phi_\rho(\theta,u):=\frac{\vartheta(\rho\e^{i\theta},u)}{h(\rho)}, \qquad \theta\in\R, \ u\in\T.
\end{equation}
We next fix a logarithm of \(\Phi_\rho\) on the whole set
\(\R\times\T\).  As in \eqref{e034}, let us put
\begin{equation}\label{hlog}
 \mathcal H(z):=
 -\sum_{j,k\geq1}\frac{z^{2jk}}{k}
 +2\sum_{j,k\geq1}(-1)^{k+1}
 \frac{z^{(2j-1)k}}{k},
 \qquad \abs z<1.
\end{equation}
Both series in \eqref{hlog} converge locally uniformly in the unit disk, and thus $ \mathcal H$ is analytic in the unit disk. Via a similar argument as in the proof  \eqref{ff01}, we obtain
\[
 \exp\mathcal H(z)=h(z),
 \qquad \mathcal H(\rho)=\log h(\rho).
\]
Together with \eqref{ff01}, this allows us to define
\begin{equation}\label{plog}
 \log\Phi_\rho(\theta,u):=
 \mathcal H(\rho\e^{i\theta})-\mathcal H(\rho)
 +\mathcal L(\rho\e^{i\theta},u),
  \qquad \theta\in\R, \ u\in\T.
\end{equation}
Note that, by \eqref{ff01} and \eqref{hlog} we have
\[
 \exp\bigl(\log\Phi_\rho(\theta,u)\bigr)
 =\frac{h(\rho\e^{i\theta})}{h(\rho)}
 \frac{\vartheta(\rho\e^{i\theta},u)}
 {h(\rho\e^{i\theta})}
 =\Phi_\rho(\theta,u).
\]
The branch in \eqref{plog} is \(2\pi\)-periodic in \(\theta\) and
\(1\)-periodic in \(u\).  Moreover,
\[
 \log\Phi_\rho(0,u)=-\ell_\rho(u),
 \qquad
 \log\Phi_\rho(\theta,0)
 =\mathcal H(\rho\e^{i\theta})-\mathcal H(\rho).
\]

The following elementary
estimate will be be useful.

\begin{lemma}\label{fgap}
For every \(0<\rho<1\), \(\theta\in\R\), and \(u\in\T\),
\begin{equation}\label{e069}
 1-\abs{\Phi_\rho(\theta,u)}^2
 \geq\frac{4\rho}{h(\rho)^2}
 \big(1-\cos \theta\cos(2\pi u)\big).
\end{equation}
\end{lemma}

\begin{proof}
We prove \eqref{e069} by a direct algebraic expansion.  Put $W_k:=\theta k^2+2\pi uk$.
Since \(h(\rho)=\sum_ka_k\), then using the symmetry under \((k,l)\mapsto(l,k)\) we have
\begin{equation*}
 1-\abs{\Phi_\rho(\theta,u)}^2
 =\frac1{h(\rho)^2}\sum_{k,l} \rho^{k^2+l^2}
 \big(1-\cos(W_k-W_l)\big).
\end{equation*}
Every summand on the right is nonnegative.  Restricting the sum to the ordered
pairs \((0,1)\), \((1,0)\), \((0,-1)\), and \((-1,0)\) gives
\[
 \frac{2\rho}{h(\rho)^2}
 \big(2-\cos(\theta+2\pi u)-\cos(\theta-2\pi u)\big),
\]
which equals exactly the right side of \eqref{e069}.
\end{proof}

From Lemma~\ref{fgap} and Proposition~\ref{invp}, we can derive the following proposition, which will be used to in the next subsection.

\begin{proposition}\label{twow}
Let \(n,d\in\N\).  There are absolute \(\delta,c>0\) such that, with
\(\rho=\rho_{n,d}\) and \(L=\min(n,d)\), the following statements hold.
For \(\nu\in\{0,1\}\) and every \(\theta\in\R\) satisfying
\begin{equation}\label{e071}
 \dist(\theta,\nu\pi+2\pi\Z)\leq\frac{\delta}{\tau_\rho},
\end{equation}
one has, uniformly in \(\xi\in\T^d\),
\begin{equation}\label{e072}
 \prod_{j=1}^d\abs{\Phi_\rho(\theta,\xi_j)}
 \leq\exp\left(-cL\tau_\rho^2
 \dist(\theta,\nu\pi+2\pi\Z)^2-cU_{\nu,\rho}(\xi)\right),
\end{equation}
where
\begin{equation*}
 U_{\nu,\rho}(\xi):=\sum_{j=1}^d
 \ell_\rho(\xi_j+\nu/2).
\end{equation*}
If  \(\theta\)
lies outside both arcs in \eqref{e071}, or equivalently,  \(\dist(\theta,\pi\Z)>\delta/\tau_\rho\), then
\begin{equation}\label{e074}
 \sup_{\xi\in\T^d}\prod_{j=1}^d
 \abs{\Phi_\rho(\theta,\xi_j)}\leq\e^{-cL}.
\end{equation}
\end{proposition}

\begin{proof}
We first choose the constants in the local and global estimates consistently.
Choose \(s_*\) so that the dense-range estimates in Lemma~\ref{trat} hold for
every \(s\geq s_*\), and then take \(\delta\) smaller than every local
constant fixed earlier.  In the dense range,
\(c_1s\leq\tau_\rho\leq c_2s\).  We may therefore choose a fixed
\(0<\delta_0\leq\delta_*\), depending only on \(\delta\), such that
\(\delta_0/s\leq\delta/\tau_\rho\) for every \(s\geq s_*\).  We invoke
Proposition~\ref{invp} with this \(\delta_0\) and fix the resulting
\(\varepsilon_0\).  Thus the local and minor arcs use compatible constants.

We begin with the proof of the local estimate  \eqref{e072}.  Write
\(\theta=\nu\pi+v\) modulo \(2\pi\) with $\nu\in\{0,1\}$ and $|v|\le \delta/\tau_\rho$. Note that a direct verification leads to  the identity
\begin{equation}\label{e075}
 \vartheta(-z,u)=\vartheta(z,u+1/2).
\end{equation}
Then we have the exact factorization
\begin{equation}\label{e116}
\prod_{j=1}^d\abs{\Phi_\rho(\theta,\xi_j)} = \prod_{j=1}^d\abs{\Phi_\rho(\nu\pi+v,\xi_j)}
 =\abs{\Phi_\rho(v,0)}^d
 \prod_{j=1}^d\abs{r_\rho(v,\xi_j+\nu/2)},
\end{equation}
for both $\nu=0$ and $\nu=1$. 
Owing to Lemma~\ref{trat}, we can bound the second factor by
\(\exp(-U_{\nu,\rho}(\xi)/2)\).  For the first factor, we shall  make use of the logarithmic branch chosen in
\eqref{plog}.  Since \(\mathcal L(z,0)=0\) by  \eqref{e034}, formula \eqref{plog} gives
\[
 \log\Phi_\rho(v,0)
 =\mathcal H(\rho\e^{iv})-\mathcal H(\rho).
\]
Consequently, a differentiation at \(v=0\) yields
\begin{equation*}
 \left.\frac{\dd}{\dd v}d\log\Phi_\rho(v,0)\right|_{v=0}
 =idD\log h(\rho)=id\,m(\rho),
\end{equation*}
which is purely imaginary and therefore does not
contribute to the logarithm of the modulus.  The
second derivative equals
\begin{equation*}
 \left.\frac{\dd^2}{\dd v^2}
 d\log\Phi_\rho(v,0)\right|_{v=0}
 =-dD^2\log h(\rho).
\end{equation*}
Following a similar line of reasoning as in the proof of Lemma  \ref{trat}, in the compact range, three termwise differentiations of the logarithmic
series obtained from the Jacobi product \eqref{e020} produce the bound
\(Cd\rho\); while in the dense range,
we differentiate \eqref{e038} with \(u=0\) and use \eqref{e037}; the main
term has size \(Cd s^3\), with an exponentially smaller transformed theta remainder.  In both cases, the third derivative on the local arc $|v|\le \delta/\tau_\rho$
is therefore bounded by \(Cd\,a(\rho)\tau_\rho^2\).  The estimate
\(dD^2\log h(\rho)\approx L\tau_\rho^2\) in \eqref{e017} identifies the
quadratic scale.
Consequently, from Taylor's formula with integral remainder it follows that
\begin{equation}\label{e126}
 \Rea\big(d\log\Phi_\rho(v,0)\big)
 =-\frac{dD^2\log h(\rho)}2v^2
 +O\big(d\,a(\rho)\tau_\rho^2\abs{v}^3\big).
\end{equation}
By \eqref{e017}, the remainder term in \eqref{e126}, divided by
\(dD^2\log h(\rho)v^2\), is \(O(\tau_\rho\abs{v})\).  After a further
decrease of \(\delta\),
through \eqref{e126} we can get
\begin{equation*}
 \abs{\Phi_\rho(v,0)}^d
 \leq\exp(-cL\tau_\rho^2v^2).
\end{equation*}
Combining this with \eqref{e116} proves \eqref{e072}, for both
values of \(\nu\).

We are left to establish a fixed gap away from these arcs.
For \(\rho\leq\rho_*\), we have \(\tau_\rho\approx1\), and every point
outside the two arcs has \(1-\abs{\cos \theta}\geq c\).  Since
\[
 1-\cos\theta\cos(2\pi u)\geq1-\abs{\cos\theta},
\]
then Lemma~\ref{fgap}, together with the uniform boundedness of \(h(\rho)\) on
this range, implies that \(1-\abs{\Phi_\rho(\theta,u)}^2\geq c\rho\).
The inequality \(\sqrt{1-x}\leq\e^{-x/2}\, (0<x<1)\), after adjusting the absolute
constant, then shows that
\(\sup_u\abs{\Phi_\rho(\theta,u)}\leq\e^{-c\rho}\).  This, together with the fact that \(d\rho\approx L\) by \eqref{e017},  establishes \eqref{e074}
in the compact range.  In the dense range, write
\(\rho=\e^{-\pi/s}\), where \(s\approx\tau_\rho\).  If \(\theta\) is
outside both arcs, then \(h(\rho)=Z_s\), and the definitions
\eqref{e052}, \eqref{e053}, and \eqref{e068} give
\(\Phi_\rho(\theta,u)=\Phi_s(\theta,u)\).  We may therefore apply \eqref{e056}
uniformly to every coordinate \(u=\xi_j\).  Therefore
\[
 \prod_{j=1}^d\abs{\Phi_\rho(\theta,\xi_j)}
 \leq(1-\varepsilon_0)^d\leq\e^{-\varepsilon_0d}.
\]
Since \(L\lesssim d\) by  \eqref{e017}, we obtain \eqref{e074}.
\end{proof}

\subsection{Weighted coefficient extraction and Gaussian kernel bounds}\label{ssad}

Having localized the coefficient integral, we now evaluate its contribution
on either major arc.  The amplitude may depend on every Fourier coordinate.

Let \(n,d\in\N\), put \(\rho:=\rho_{n,d}\), and define
\begin{equation}\label{e085}
 I_0:=[-\delta/\tau_\rho,\delta/\tau_\rho],
\end{equation}
where \(\delta\) is fixed as in Proposition~\ref{twow}.  Throughout this
subsection, the real variable \(\theta\) belongs to the interval \(I_0\) in
\eqref{e085}, and all local moments and weighted integrals are taken over this
interval.  The logarithm of \(h(\rho\e^{i\theta})\) is the branch
\(\mathcal H(\rho\e^{i\theta})\) defined in \eqref{hlog}.

We set
\begin{equation}\label{e076}
 \Psi(\theta):=\e^{-in\theta}\frac{1-\rho}{1-\rho\e^{i\theta}}
 \left(\frac{h(\rho\e^{i\theta})}{h(\rho)}\right)^d=\e^{K(\theta)},
\end{equation}
where \(K(0)=0\).  Explicitly, for \(\theta\in I_0\),
\begin{equation}\label{e077}
 \begin{aligned}
 K(\theta):={}&-in\theta
 +d\bigl(\mathcal H(\rho\e^{i\theta})-\mathcal H(\rho)\bigr)
 +\log(1-\rho)
 +\sum_{j\geq1}\frac{\rho^j\e^{ij\theta}}{j}.
 \end{aligned}
\end{equation}
The last series in \eqref{e077} is absolutely convergent and equals
\(-\log(1-\rho\e^{i\theta})\).  Hence \eqref{e076} holds and \(K(0)=0\).  Notice
that \(K\) is a logarithmic lift on \(\R\), rather than a periodic logarithm:
\(K(\theta+2\pi)=K(\theta)-2\pi in\).  Only its restriction to \(I_0\) will be
used below.
The saddle-point equation \(n=dm(\rho)+a(\rho)\) in \eqref{e008} gives
\begin{equation}\label{e078}
 K'(0)=0,\qquad K''(0)=-V,
 \qquad
 K'''(0)=-i\left(dD^3\log h(\rho)
 +\frac{\rho(1+\rho)}{(1-\rho)^3}\right).
\end{equation}

On the \(0\)-centered major arc, we use the following uniform complex
derivative bounds.

\begin{lemma}\label{kder}
For every fixed \(r\in\{3,4\}\) and
\(\abs{\theta}\leq\delta/\tau_\rho\),
\begin{equation}\label{e079}
 \abs{K^{(r)}(\theta)}\lesssim_rL\tau_\rho^r,
 \qquad \Rea K(\theta)\leq-cV\theta^2.
\end{equation}
\end{lemma}

\begin{proof}
We again consider the two parameter ranges separately.  In the
compact range, differentiate the absolutely
convergent logarithmic series arising from \eqref{e020}.  As in the
compact-range proof of \eqref{e025}, the \(r\)-th derivative of the theta logarithm is
\(O_r(\rho)\), uniformly on the fixed arc.  The   term
\(-\log(1-\rho\e^{i\theta})\) in \(K(\theta)\) in \eqref{e077} has the
series
\begin{equation}\label{e090}
 -\log(1-\rho\e^{i\theta})
 =\sum_{j\geq1}\frac{\rho^j\e^{ij\theta}}j,
\end{equation}
whose \(r\)-th derivative is bounded by
\(\sum_{j\geq1}j^{r-1}\rho^j\lesssim_r\rho\).  This last series comes from
the geometric series \eqref{e090}.  Thus the theta contribution to \(K^{(r)}\) is at most
\(C_rd\rho\), and the geometric contribution is at most \(C_r\rho\).  By
Lemma~\ref{spar}, their sum is
\(O_r(L\tau_\rho^r)\).

In the dense range, use \(S(\theta)\) from \eqref{e036}.  The complex form
of \eqref{e023} is
\begin{equation*}
 h(\rho\e^{i\theta})=S(\theta)^{1/2}h(\e^{-\pi S(\theta)}).
\end{equation*}
The logarithm of the right side is
\begin{equation}\label{e162}
 \tfrac12\log S(\theta)+\log h(\e^{-\pi S(\theta)}).
\end{equation}
Formula \eqref{e037} shows directly that the \(r\)-th derivative of the
first term in \eqref{e162} is \(O_r(s^r)\).  Since \(\Rea S\approx s\),
termwise differentiation of the second term in \eqref{e162} gives
\(O_r(s^{2r}\e^{-cs})\), which is smaller.  The geometric term is bounded
by
\begin{equation*}
 \sum_{j\geq1}j^{r-1}\rho^j\lesssim_rs^r.
\end{equation*}
Since \(L \approx d\) and \(s\approx\tau_\rho\), this proves the first estimate.
For the real part, Taylor's formula, \eqref{e078}, and the case \(r=3\)
give
\begin{equation}\label{e091}
 \Rea K(\theta)=-\frac{V\theta^2}{2}
 +O(L\tau_\rho^3\abs{\theta}^3).
\end{equation}
Because \(V\approx L\tau_\rho^2\), the remainder in \eqref{e091} is at
most one quarter of \(V\theta^2\) after we decrease \(\delta\).  This proves the
second estimate.
\end{proof}

We formulate the local expansion with the normalization used in coefficient
extraction.

\begin{lemma}\label{edge}
There is an absolute constant \(L_1\geq1\) such that the following holds.
For \(q\geq0\), let
\begin{equation}\label{e092}
 J_q:=\int_{I_0}\theta^q\Psi(\theta)\dd\theta.
\end{equation}
If \(L\geq L_1\),
then
\begin{equation}\label{e082}
 J_0=\sqrt{\frac{2\pi}{V}}\big(1+O(L^{-1})\big),
\end{equation}
and, for every fixed integer \(q\geq0\),
\begin{equation}\label{e133}
 \frac{\int_{I_0}\abs{\theta}^q\abs{\Psi(\theta)}\dd\theta}{\abs{J_0}}
 \lesssim_q V^{-q/2}.
\end{equation}
\end{lemma}

\begin{proof}
We separate the central region from the part controlled by Gaussian decay.
We combine Lemma~\ref{kder} with Taylor's formula and its integral
remainder to obtain
\begin{equation}\label{e086}
 K(\theta)=-\frac{V\theta^2}{2}+\frac{K'''(0)\theta^3}{6}
 +R_4(\theta),\qquad
 \abs{R_4(\theta)}\lesssim L\tau_\rho^4\abs{\theta}^4.
\end{equation}
Put
\begin{equation}\label{e100}
 y:=\sqrt{V}\theta,\qquad
 \alpha:=\frac{K'''(0)}{V^{3/2}},\qquad
 w(y):=\frac{\alpha y^3}{6}+R_4(y/\sqrt V).
\end{equation}
The derivative estimate \eqref{e079} and the remainder bound in
\eqref{e086} yield
\begin{equation}\label{e087}
 \abs{\alpha}\lesssim L^{-1/2},\qquad
 \abs{R_4(y/\sqrt V)}\lesssim\frac{y^4}{L}.
\end{equation}
For \(\abs{y}\leq L^{1/16}\), we expand the exponential and find
\begin{equation}\label{e088}
 \Psi(y/\sqrt V)=\e^{-y^2/2}
 \left(1+\frac{\alpha y^3}{6}+E_L(y)\right),\qquad
 \abs{E_L(y)}\lesssim\frac{y^4+y^6}{L}.
\end{equation}
Indeed, on this interval both terms in \(w(y)\) from \eqref{e100} have
modulus at most a fixed small constant.  More precisely, \eqref{e087} and
\(\abs y\leq L^{1/16}\) give
\[
 \abs{\alpha y^3}\lesssim L^{-5/16},\qquad
 \abs{R_4(y/\sqrt V)}\lesssim L^{-3/4}.
\]
For large \(L\), the elementary estimate
\(\abs{\e^w-1-w}\lesssim\abs{w}^2\) is therefore uniform, and
\[
 \abs{\e^w-1-\alpha y^3/6}
 \lesssim\abs{R_4(y/\sqrt V)}+\abs{\alpha}^2y^6.
\]
This proves the error bound in \eqref{e088} directly from \eqref{e087}.
The endpoints of the rescaled \(0\)-centered major arc have size comparable to
\(\sqrt L\).  On
\(L^{1/16}\leq\abs y\lesssim\sqrt L\), Lemma~\ref{kder} shows that
\(\abs{\Psi(y/\sqrt V)}\leq\e^{-cy^2}\).  Consequently, the integral of
any fixed power of \(y\) over this region is \(O(\e^{-c'L^{1/8}})\).
The same Gaussian bound allows us to extend the inner integral from
\(\abs y\leq L^{1/16}\) to the real line.  For all fixed integers
\(q,M\geq0\),
the omitted \(q\)-th moment is \(O_{q,M}(L^{-M})\).  In particular, these
terms are smaller than the error bounds in the statement.

The change of variables \(y=\sqrt V\theta\) in \eqref{e100} gives
\begin{equation}\label{e101}
 J_q=V^{-(q+1)/2}\int y^q\Psi(y/\sqrt V)\dd y,
\end{equation}
with the integral initially over the rescaled arc.  The symmetry
\(K(-\theta)=\overline{K(\theta)}\) makes \(J_0\) real.  In
\eqref{e088}, the term proportional to \(y^3\) is odd and integrates to
zero, whereas the integral of \(E_L\) is \(O(L^{-1})\).  This proves
\eqref{e082} and, in particular, shows that \(J_0\) is positive for large
\(L\).  The Gaussian majorant gives \eqref{e133} after the change of variables
in \eqref{e101}.  The odd cubic term vanishes by symmetry, which improves the
relative error in \eqref{e082} from order
\(L^{-1/2}\) to order \(L^{-1}\).
\end{proof}

For \(1<p<2\), we need the following version of arbitrary fixed order.  Its
coefficients are the exact local moments.

\begin{proposition}\label{hwt}
Let \(N\geq0\) be fixed.  Suppose that \(L\geq L_1\), with \(L_1\)
chosen as in Lemma~\ref{edge}, and let \(U\geq0\).  Assume that
\(\mathbb{A}\in C^{N+1}(I_0;\mathbb C)\), with
\begin{equation}\label{e104}
 \abs{\mathbb{A}^{(j)}(\theta)}
 \leq C_j\tau_\rho^j\e^{-cU},
 \qquad 0\leq j\leq N+1,
\end{equation}
where \(c>0\) and the constants \(C_j\) are independent of \(n\), \(d\),
and \(U\).  With the moments \(J_j\) defined in \eqref{e092}, set
\begin{equation}\label{e134}
 b_j:=\frac{J_j}{j!J_0},\qquad j\geq0.
\end{equation}
Then
\begin{equation}\label{e135}
 \abs{b_j}\lesssim_j V^{-j/2},\qquad j\geq0,
\end{equation}
and
\begin{equation}\label{e136}
 \frac{\int_{I_0}\Psi(\theta)\mathbb{A}(\theta)\dd\theta}{J_0}
 =\sum_{j=0}^N b_j\mathbb{A}^{(j)}(0)
 +O_N\left(\frac{\e^{-cU}}{L^{(N+1)/2}}\right).
\end{equation}
The implicit constants may depend on \(N\), \(c\), and the \(C_j\)'s, but
not on \(n\), \(d\), or \(U\).
\end{proposition}

\begin{proof}
Estimation \eqref{e135} follows from \eqref{e133}.  Taylor's formula
with integral remainder and the hypothesis \eqref{e104} give
\begin{equation}\label{e105}
 \mathbb{A}(\theta)=\sum_{j=0}^N\frac{\mathbb{A}^{(j)}(0)}{j!}\theta^j
 +O_N\big(
 \tau_\rho^{N+1}\e^{-cU}
 \abs{\theta}^{N+1}\big).
\end{equation}
We integrate \eqref{e105} against \(\Psi\), divide by \(J_0\), and apply
\eqref{e133} with \(q=N+1\).  The normalized remainder is at most
\[
 C_N\tau_\rho^{N+1}V^{-(N+1)/2}
 \e^{-cU}.
\]
Since \(V\approx L\tau_\rho^2\) by \eqref{e018}, this is the error in
\eqref{e136}.
\end{proof}

We are now in a position to verify the amplitude hypothesis  \eqref{e104}, which appears in Proposition  \ref{hmul}. Recall  \eqref{e049}. For
\(\nu\in\{0,1\}\), we define
\begin{equation}\label{e093}
 \mathbb{A}_{0,\xi}(\theta):=\prod_{j=1}^dr_\rho(\theta,\xi_j),\qquad
 \mathbb{A}_{1,\xi}(\theta):=R_\rho(\theta)
 \prod_{j=1}^dr_\rho(\theta,\xi_j+1/2).
\end{equation}

\begin{lemma}\label{ader}
There is an absolute constant \(c>0\) such that, for every fixed integer
\(j\geq0\), every \(\nu\in\{0,1\}\), every
\(\xi\in\T^d\), and every \(\abs{\theta}\leq\delta/\tau_\rho\),
\begin{equation}\label{e050}
 \abs{R_\rho(\theta)}\approx1,\qquad
 \abs{R_\rho^{(j)}(\theta)}\lesssim_j\tau_\rho^j,
\end{equation}
and
\begin{equation}\label{e094}
 \abs{\mathbb{A}_{\nu,\xi}^{(j)}(\theta)}
 \lesssim_j\tau_\rho^j\e^{-cU_{\nu,\rho}(\xi)}.
\end{equation}
\end{lemma}

\begin{proof}
The definitions of \(a(\rho)\) and \(\tau_\rho\) in \eqref{e016} give
\((1-\rho)^{-1}=1+a(\rho)=\tau_\rho\).
On the local arc,
\(\abs{1-\rho\e^{i\theta}}\approx1-\rho\) and
\(\abs{1+\rho\e^{i\theta}}\approx1+\rho\).  Thus
\(\abs{R_\rho(\theta)}\approx1\).  Every fixed derivative of positive order
of
\begin{equation}\label{e106}
 \log R_\rho(\theta)= \log\frac{1+\rho}{1-\rho} +\log(1-\rho\e^{i\theta})
 -\log(1+\rho\e^{i\theta})
\end{equation}
is bounded by \(C_j\tau_\rho^j\).  Applying the Fa\`a di Bruno formula to
\(R_\rho(\theta)=\exp(\log R_\rho(\theta))\), with the logarithm given in
\eqref{e106}, proves
the derivative estimate in \eqref{e050}.

All factors in \eqref{e093} are nonzero on the local arc: this follows
from the Jacobi triple product \eqref{e019} for the theta ratios and from the
definition of \(R_\rho(\theta)\) in \eqref{e049}.  We may therefore choose their
logarithms by continuation from \(\theta=0\).
Lemma~\ref{trat} implies
\begin{equation}\label{e095}
 \abs{\mathbb{A}_{\nu,\xi}(\theta)}\lesssim\e^{-cU_{\nu,\rho}(\xi)},
 \qquad
 \abs{\partial_\theta^k\log \mathbb{A}_{\nu,\xi}(\theta)}
 \lesssim_k\tau_\rho^k(1+U_{\nu,\rho}(\xi)).
\end{equation}
For \(\nu=1\), estimate \eqref{e050} controls the additional rational
factor.  Applying the Fa\`a di Bruno formula to
\(\mathbb{A}_{\nu,\xi}=\exp(\log\mathbb{A}_{\nu,\xi})\) gives a factor
bounded by \(C_j\tau_\rho^j(1+U_{\nu,\rho}(\xi))^j\).  Multiplication by
the first estimate in \eqref{e095} completes the proof, since, for every
\(c_0>0\),
\begin{equation}\label{e107}
 (1+U)^j\e^{-c_0U}\lesssim_{j,c_0}\e^{-c_0U/2}
 \qquad(U\geq0).
\end{equation}
Applying \eqref{e107} proves \eqref{e094} after renaming the exponent
constant.
\end{proof}

We shall compare the Taylor kernels derived from Proposition  \ref{hmul}  with the normalized discrete Gaussians.  We use the
following maximal estimate of Mirek, Szarek, and
Wr\'obel~\cite[(1.5) of Theorem 1.1]{MSW25}.

\begin{theorem}[Mirek--Szarek--Wr\'obel]
\label{gaus}
For \(s>0\), define
\begin{equation}\label{e027}
 g_s(x):=\frac{\e^{-\pi\abs{x}^2/s}}
 {\sum_{k\in\Z^d}\e^{-\pi |k|^2/s}},
 \qquad G_sf:=g_s*f.
\end{equation}
Then for every \(p\in(1,\infty)\), there is \(C_p>0\) independent of
\(d\), such that, for any \(d\in\N\) and \(f\in\ell^p(\Z^d)\),
\begin{equation}\label{e028}
 \norm{\sup_{s>0}\abs{G_sf}}_{\ell^p(\Z^d)}
 \leq C_p \,\norm{f}_{\ell^p(\Z^d)}.
\end{equation}
\end{theorem}

We now identify the kernels of the Taylor terms.  The resulting estimate
allows us to use an expansion of arbitrary fixed order.  For
\(0<\rho<1\), set
\begin{equation}\label{e137}
 \gamma_\rho(x):=\frac{\rho^{\abs{x}^2}}{h(\rho)^d},\qquad x\in\Z^d.
\end{equation}
Thus \(\gamma_{\exp(-\pi/s)}=g_s\), where \(g_s\) is defined in
\eqref{e027}.

\begin{lemma}\label{tilt}
There is an absolute constant \(\eta>0\) with the following property.  Let
\(n,d\in\N\), \(\rho=\rho_{n,d}\), \(L=\min(n,d)\), and put
\begin{equation*}
 \eta_V:=\frac{\eta}{\sqrt V},\qquad
 \rho_\pm:=\rho\e^{\pm\eta_V},\qquad
 W_\rho(x):=\abs{x}^2-dm(\rho).
\end{equation*}
Then \(0<\rho_-<\rho<\rho_+<1\), and, for every fixed integer \(j\geq0\),
\begin{equation}\label{e139}
 \left(1+\frac{\abs{W_\rho(x)}}{\sqrt V}\right)^j\gamma_\rho(x)
 \lesssim_j\gamma_{\rho_-}(x)+\gamma_{\rho_+}(x),
 \qquad x\in\Z^d.
\end{equation}
\end{lemma}

\begin{proof}
Since \(V\approx L\tau_\rho^2\), we have
\(\eta_V\lesssim\eta/(\tau_\rho\sqrt L)\leq C\eta/\tau_\rho\).  We
fix \(\eta\) so small that \(\eta_V\leq(1-\rho)/4\).  Since
\(-\log\rho\geq1-\rho\), this choice gives \(\eta_V<-\log\rho\) and
hence \(\rho_+<1\).  It also makes \(\eta_V\) a
sufficiently small multiple of \(1-\rho\), uniformly for every \(L\geq1\).

For \(\abs u\leq\eta_V\), define \(\rho_u\) and record the comparisons
needed below as follows:
\begin{equation}\label{e140}
 \rho_u:=\rho\e^u,\qquad \abs u\leq\eta_V,\qquad
 a(\rho_u)\approx a(\rho),\qquad
 \tau_{\rho_u}\approx\tau_\rho.
\end{equation}
Indeed, \(\eta_V\) is a sufficiently small multiple of \(1-\rho\), so
\(\rho_u\) and \(\rho\), as well as \(1-\rho_u\) and \(1-\rho\), are
comparable.  This proves the comparisons in \eqref{e140}.
Using the second estimate in \eqref{e025}, \eqref{e140}, and
\(d\rho\approx L\), we obtain
\begin{equation}\label{e141}
 dD^2\log h(\rho_u)
 \lesssim d\,a(\rho)\tau_\rho\approx V,
 \qquad \abs u\leq\eta_V.
\end{equation}

Define
\begin{equation}\label{e110}
 F(u):=d\log h(\rho\e^u),\qquad
 F'(0)=dm(\rho),\qquad
 F''(u)=dD^2\log h(\rho\e^u)\geq0.
\end{equation}
Taylor's formula with integral remainder gives, for
\(\sigma\in\{-1,1\}\),
\begin{equation}\label{e111}
 F(\sigma\eta_V)-F(0)-\sigma\eta_VF'(0)
 =\eta_V^2\int_0^1(1-v)F''(\sigma v\eta_V)\dd v.
\end{equation}
The integral in \eqref{e111} is nonnegative by \eqref{e110}, and
\eqref{e141} bounds it by \(CV\).
Consequently,
\[
 0\leq F(\sigma\eta_V)-F(0)-\sigma\eta_VF'(0)
 \lesssim V\eta_V^2\lesssim1.
\]
The definition of \(\gamma_\rho\) in \eqref{e137} gives
\begin{equation}\label{e142}
 \e^{\sigma\eta_VW_\rho(x)}\gamma_\rho(x)
 =\exp\big(F(\sigma\eta_V)-F(0)-\sigma\eta_VF'(0)\big)
 \gamma_{\rho\e^{\sigma\eta_V}}(x)
 \lesssim\gamma_{\rho\e^{\sigma\eta_V}}(x).
\end{equation}
Finally,
\((1+\abs y)^j\lesssim_j\e^{\eta y}+\e^{-\eta y}\).  Apply this with
\(y=W_\rho(x)/\sqrt V\) and use \eqref{e142} to obtain \eqref{e139}.
\end{proof}

For \(\nu\in\{0,1\}\) and \(j\geq0\), define
\begin{equation}\label{e112}
 K_{\nu,j,\rho}(x)
 :=\mathcal F^{-1}\!\left[\mathbb{A}_{\nu,\cdot}^{(j)}(0)\right](x)
 =\int_{\T^d}\mathbb{A}_{\nu,\xi}^{(j)}(0)
   \es{-x\cdot\xi}\dd\xi.
\end{equation}
Thus \(K_{\nu,j,\rho}\) is the convolution kernel of the multiplier
\(\mathbb{A}_{\nu,\xi}^{(j)}(0)\) in \eqref{e112}.

\begin{lemma}\label{jker}
Let \(n,d\in\N\), \(\rho=\rho_{n,d}\), \(L=\min(n,d)\), and let \(b_j\)
be as in \eqref{e134}.  For every fixed integer \(j\geq0\), if
\(L\geq L_1\),
\begin{equation}\label{e143}
 \abs{b_jK_{\nu,j,\rho}(x)}
 \lesssim_j\gamma_{\rho_-}(x)+\gamma_{\rho_+}(x),
 \qquad x\in\Z^d,\quad \nu\in\{0,1\}.
\end{equation}
If \(Q_{\nu,j,n,d}\) denotes convolution operator by the kernel
\(b_jK_{\nu,j,\rho}\), then, for every \(\nu\in\{0,1\}\), every
\(1<p<\infty\), and every \(f\in\ell^p(\Z^d)\),
\begin{equation}\label{e145}
 \norm{\sup_{n:\,\min(n,d)\geq L_1}
 \abs{Q_{\nu,j,n,d}f}}_p\lesssim_{p,j}\norm{f}_p.
\end{equation}
\end{lemma}

\begin{proof}
Termwise Fourier inversion and differentiation are justified by the absolute
exponential convergence of the theta series.  More precisely,
\begin{equation}\label{e113}
 \mathcal F^{-1}\!\left[\mathbb{A}_{0,\cdot}(\theta)\right](x)
 =\gamma_\rho(x)\exp\left(
 i\theta\abs{x}^2
 -d\big(\log h(\rho\e^{i\theta})-\log h(\rho)\big)
 \right).
\end{equation}
For the exponential factor in \eqref{e113}, define \(\Lambda_x\) and record
its derivatives at zero as follows:
\begin{equation}\label{e129}
 \begin{aligned}
 \Lambda_x(\theta)
 &:=i\theta\abs{x}^2
 -d\big(\log h(\rho\e^{i\theta})-\log h(\rho)\big),\\
 \Lambda_x'(0)&=iW_\rho(x),\\
 \Lambda_x^{(r)}(0)&=-i^rc_r,\qquad r\geq2,\\
 c_r&:=dD^r\log h(\rho),\qquad r\geq2.
 \end{aligned}
\end{equation}
Estimations \eqref{e025} and \eqref{e017} imply that
\begin{equation}\label{e146}
 \abs{c_r}\lesssim_r L\tau_\rho^r\lesssim_r V^{r/2},
 \qquad r\geq2.
\end{equation}
The Fa\`a di Bruno formula, applied to \(\exp(\Lambda_x(\theta))\) in
\eqref{e113} with
the logarithmic derivatives in \eqref{e129}, and combined with
\eqref{e146}, yields
\begin{equation}\label{e147}
 \abs{K_{0,j,\rho}(x)}
 \lesssim_j V^{j/2}
 \left(1+\frac{\abs{W_\rho(x)}}{\sqrt V}\right)^j\gamma_\rho(x).
\end{equation}
Combining \eqref{e135}, \eqref{e139}, and \eqref{e147} proves
\eqref{e143} for \(\nu=0\).

Let \(\chi(x):=(-1)^{x_1+\cdots+x_d}\).  The Leibniz rule and the Fourier
shift by \(1/2\) give
\begin{equation}\label{e132}
 K_{1,j,\rho}(x)=\chi(x)
 \sum_{l=0}^j\binom jl R_\rho^{(l)}(0)K_{0,j-l,\rho}(x).
\end{equation}
For the summand indexed by \(l\) in \eqref{e132}, estimates \eqref{e050}, \eqref{e135}, and
\eqref{e147} give the bound
\[
 C_j\left(\frac{\tau_\rho}{\sqrt V}\right)^l
 \left(1+\frac{\abs{W_\rho(x)}}{\sqrt V}\right)^{j-l}
 \gamma_\rho(x).
\]
Since \(\tau_\rho/\sqrt V\lesssim L^{-1/2}\leq1\), Lemma~\ref{tilt}
proves \eqref{e143} for \(\nu=1\).

Choose \(s_\pm>0\) through \(\rho_\pm=\exp(-\pi/s_\pm)\).
Equation \eqref{e143} implies the pointwise operator estimate
\begin{equation}\label{e144}
 \abs{Q_{\nu,j,n,d}f}
 \lesssim_jG_{s_-}\abs f+G_{s_+}\abs f.
\end{equation}
Taking the supremum in \eqref{e144} and applying the maximal estimate
\eqref{e028} proves \eqref{e145}.
\end{proof}

\section{Proof of Theorem~\ref{mthm}}\label{sprf}

\subsection{Uniform asymptotic expansions for the multiplier $\mathfrak m_{\sqrt{n}}$}\label{sapp}

We now combine the coefficient estimates.  We decompose the Cauchy integral
into two major arcs and the complementary minor arc.  After estimating all
three integrals, we divide by the counting coefficient.  Fix \(A>1\),
to be chosen in the next subsection, and assume for now that
\begin{equation}\label{e096}
 1\leq n\leq Ad^2,\qquad L=\min(n,d)\geq L_0(A).
\end{equation}
Here \(L_0(A)\geq L_1\) is a preliminary threshold used only for the
normalization argument below; in Proposition~\ref{hmul} we increase it as
needed for the chosen expansion order.  Let \(\rho:=\rho_{n,d}\).
Cauchy's formula expresses the numerator in
\eqref{e007} as
\begin{equation*}
 \frac{\rho^{-n}h(\rho)^d}{2\pi(1-\rho)}
 \int_{-\pi}^{\pi}\mathcal I_\xi(\theta)\dd \theta.
\end{equation*}
The same prefactor occurs when \(\xi=0\).  It is positive and cancels,
so the coefficient ratio is exactly
\begin{equation}\label{e097}
 \mathfrak m_{\sqrt n}(\xi)=
 \frac{\int_{-\pi}^{\pi}\mathcal I_\xi(\theta)\dd \theta}
 {\int_{-\pi}^{\pi}\mathcal I_0(\theta)\dd \theta},
\end{equation}
where
\begin{equation*}
 \mathcal I_\xi(\theta):=\e^{-in\theta}
 \frac{1-\rho}{1-\rho\e^{i\theta}}
 \prod_{j=1}^d\frac{\vartheta(\rho\e^{i\theta},\xi_j)}{h(\rho)}.
\end{equation*}

We now develop the asymptotic expansion of an arbitrary prescribed order for the multiplier $ \mathfrak m_{\sqrt n}(\xi)$ uniformly in $\xi\in\T^d$, when both $n$ and $d$ are large enough.

\begin{proposition}\label{hmul}
Fix \(A>1\) and an integer \(N\geq0\).  There exist an integer
\(L_0=L_0(A,N)\) and a constant \(C_{A,N}\) such that, for any $n\in \mathcal N_{d,N}$ and
 any
\(\xi\in\T^d\),
\begin{equation}\label{e150}
 \mathfrak m_{\sqrt n}(\xi)=p_{n,d}^{[N]}(\xi)+r_{n,d}^{[N]}(\xi),
\end{equation}
where $\mathcal N_{d,N}$ denotes the set of $n\in\N$ satisfying  \eqref{e096} and
\begin{equation}\label{e151}
 p_{n,d}^{[N]}(\xi):=\sum_{j=0}^Nb_j
 \left(\mathbb{A}_{0,\xi}^{(j)}(0)
 +(-1)^n\omega_\rho\mathbb{A}_{1,\xi}^{(j)}(0)\right)
\end{equation}
and
\begin{equation}\label{e152}
 \sup_{\xi\in\T^d}\abs{r_{n,d}^{[N]}(\xi)}
 \leq C_{A,N}L^{-(N+1)/2}.
\end{equation}
Furthermore, supposing that \(P_{n,d}^{[N]}\) is the convolution operator with multiplier
\(p_{n,d}^{[N]}\), then
\begin{equation}\label{e153}
 \norm{P_{n,d}^{[N]}}_{\ell^1\to\ell^1}\lesssim_{A,N}1,
\end{equation}
and, for any \(1<p<\infty\),
\begin{equation}\label{e154}
 \norm{\sup_{n\in\mathcal N_{d,N}}
 \abs{P_{n,d}^{[N]}f}}_{\ell^p\to\ell^p}
 \lesssim_{A,N,p} 1.
\end{equation}
\end{proposition}

\begin{proof}
Increase \(L_0(A,N)\), if necessary, so that
\(L_0(A,N)\geq\max\{L_0(A),L_1\}\).
On the circle \(\R/(2\pi\Z)\), set
\begin{equation}\label{e148}
 \begin{aligned}
 \mathcal J_0
 &:=\{\dist(\theta,2\pi\Z)\leq\delta/\tau_\rho\},\\
 \mathcal J_\pi
 &:=\{\dist(\theta,\pi+2\pi\Z)\leq\delta/\tau_\rho\},\\
 \mathcal J_{\mathrm{min}}
 &:=\R/(2\pi\Z)\setminus(\mathcal J_0\cup\mathcal J_\pi).
 \end{aligned}
\end{equation}
On \(\mathcal J_0\) in \eqref{e148}, represented by
\(\abs\theta\leq\delta/\tau_\rho\), we have
\begin{equation*}
 \mathcal I_\xi(\theta)=\Psi(\theta)\mathbb{A}_{0,\xi}(\theta).
\end{equation*}

The two pieces of the arc at \(\pi\) in \([-\pi,\pi]\) form one arc on
the circle.  We parameterize it by \(\theta=\pi+v\) modulo \(2\pi\), with
\(\abs{v}\leq\delta/\tau_\rho\).  Identity \eqref{e075} gives the exact
formula
\begin{equation*}
 \mathcal I_\xi(\pi+v)=(-1)^n\omega_\rho
 \Psi(v)\mathbb{A}_{1,\xi}(v).
\end{equation*}
Indeed, the phase contributes \((-1)^n\), and the theta product is shifted
by \(1/2\).  The geometric factors satisfy
\begin{equation*}
 \omega_\rho\frac{1-\rho}{1-\rho\e^{iv}}R_\rho(v)
 =\frac{1-\rho}{1+\rho\e^{iv}}.
\end{equation*}
There is no factor two, because the two endpoint intervals are the two
parts of this single circular arc.  There is also no factor \((-1)^d\),
because \(k^2\) and \(k\) have the same parity in \eqref{e075}.

We first record the estimates needed to normalize the coefficient.  On the
complementary set, we have
\begin{equation}\label{e122}
 \frac{1-\rho}{\abs{1-\rho\e^{i\theta}}}\leq1.
\end{equation}
Together with Proposition~\ref{twow}, inequality \eqref{e122} bounds the integral over
\(\mathcal J_{\mathrm{min}}\) in \eqref{e148} by \(C\e^{-cL}\), since this set has total length at
most \(2\pi\).
By \eqref{e082}, the size of the minor-arc integral relative to \(J_0\) is at
most
\begin{equation}\label{e102}
 C\tau_\rho\sqrt L\,\e^{-cL}.
\end{equation}
Under \eqref{e096}, this quantity decays faster than any prescribed fixed
power of \(L\).  Indeed, if \(n\leq d\),
then \(\tau_\rho\approx1\).  If \(d<n\leq Ad^2\), Lemma~\ref{spar}
yields \(\tau_\rho\lesssim_A d\), while \(L=d\).  Thus, throughout the
range \eqref{e096},
\begin{equation}\label{e127}
 \tau_\rho\lesssim_A 1+L,\qquad
 \tau_\rho\sqrt L\,\e^{-cL}\lesssim_{A,M} L^{-M}
 \quad\text{for every fixed }M>0.
\end{equation}

We next normalize the coefficient.  The contribution of \(\mathcal J_0\) to the
denominator can be evaluated exactly:
\begin{equation*}
 \int_{\mathcal J_0}\mathcal I_0(\theta)\dd \theta=J_0,
\end{equation*}
because \(\mathbb{A}_{0,0}(\theta)=1\) on \(\mathcal J_0\).  Moreover,
Lemma~\ref{trat}, together with Lemma~\ref{spar}, implies
\begin{equation}\label{e103}
 U_{1,\rho}(0)=d\ell_\rho(1/2)\gtrsim L.
\end{equation}
Equation \eqref{e095}, \eqref{e103}, and the case \(q=0\) of
\eqref{e133} show that the \(\mathcal J_\pi\) contribution to the
denominator is \(O(J_0\e^{-cL})\).  By \eqref{e102} and \eqref{e127}, the
minor arc has the same form after decreasing \(c\).  Hence
\begin{equation}\label{e003}
 \frac{\int_{-\pi}^{\pi}\mathcal I_0(\theta)\dd \theta}{J_0}
 =1+O_A(\e^{-cL}).
\end{equation}
After increasing \(L_0(A,N)\), the ratio in \eqref{e003} has modulus at
least \(1/2\).

We now estimate the Taylor remainder on the two major arcs.
Apply Proposition~\ref{hwt} to \(\mathbb{A}_{0,\xi}\) and
\(\mathbb{A}_{1,\xi}\), using Lemma~\ref{ader}.  The two normalized
major-arc remainders are bounded by \(C_NL^{-(N+1)/2}\), uniformly in
\(\xi\).  The minor arc, divided by \(J_0\), has size at most the quantity
in \eqref{e102}.
In the range \(n\leq Ad^2\), estimate \eqref{e127} shows that this is
\(O_{A,N}(L^{-(N+1)/2})\).  Consequently, after division by \(J_0\), the
numerator and denominator in \eqref{e097} have the form
\begin{equation}\label{e009}
 p_{n,d}^{[N]}(\xi)+E_\xi,
 \qquad 1+F_0,
\end{equation}
respectively, where
\begin{equation}\label{e010}
 \sup_{\xi\in\T^d}\abs{E_\xi}\lesssim_{A,N}L^{-(N+1)/2},
 \qquad \abs{F_0}\lesssim_A\e^{-cL}.
\end{equation}
Equations \eqref{e094}, \eqref{e135}, and
the estimate \(V\approx L\tau_\rho^2\) in \eqref{e018} also give
\(\sup_\xi\abs{p_{n,d}^{[N]}(\xi)}\lesssim_N1\). Hence by  \eqref{e009}-\eqref{e010}, after
increasing \(L_0(A,N)\), we obtain
\eqref{e150}--\eqref{e152}, as required.

By Lemma~\ref{jker}, every summand in the inverse Fourier kernel of
\eqref{e151} has uniformly bounded \(\ell^1\) norm; here we use that each
normalized Gaussian in \eqref{e143} has total mass one.  Since the sum has
only \(2(N+1)\) terms and \(0<\omega_\rho\leq1\), this proves
\eqref{e153}.
The same lemma and
Theorem~\ref{gaus}, followed by the triangle inequality over
\(0\leq j\leq N\) and the two major arcs, prove \eqref{e154}.
\end{proof}

\subsection{Completion of the proof}\label{scom}

We now pass from the multiplier asymptotics to the maximal operator.  For
large radii we use the following theorem, whose supremum is taken over every
sufficiently large radius.

\begin{theorem}[Bourgain, Mirek, Stein, and Wr\'obel~\cite{BMS21},
  Theorem 2]
\label{larg}
There is an absolute \(C_0>0\) such that, for every
\(p\in(1,\infty]\), every \(d\in\N\), and every
\(f\in\ell^p(\Z^d)\),
\begin{equation}\label{e108}
 \norm{\sup_{t\geq C_0d}\abs{M_tf}}_{\ell^p(\Z^d)}
 \lesssim_p\norm{f}_{\ell^p(\Z^d)},
\end{equation}
with a constant independent of \(d\).
\end{theorem}

\begin{proof}[Proof of Theorem~\ref{mthm}]
We fix an \(A>1\) so that
\begin{equation}\label{e109}
 A>(C_0+1)^2.
\end{equation}
For any \(1<p\leq2\), we choose an integer \(N\geq0\) such that
\begin{equation}\label{e156}
 \alpha:=(N+1)(p-1)\geq2.
\end{equation}
Condition \eqref{e156} makes the \(p\)-th powers of the interpolated remainder
bounds summable over the squared radii.
For brevity, write \(L_0:=L_0(A,N)\) for the threshold in
Proposition~\ref{hmul}.  This number may depend on \(p\) through \(N\), but
it is independent of \(d\).
We divide the squared radii into three classes: \textbf{(i)} the set
\(\mathcal N_{d,N}\); \textbf{(ii)} the indices \(n\leq Ad^2\) with
\(\min(n,d)<L_0\); \textbf{(iii)} and the indices \(n>Ad^2\).  Proposition~\ref{hmul}
bounds the Taylor terms on the first class.  We first sum the corresponding
remainders, and then treat the other two classes.
For \(n\in\mathcal N_{d,N}\), let
\begin{equation}\label{e026}
 R_{n,d}^{[N]}:=M_{\sqrt n}-P_{n,d}^{[N]}.
\end{equation}
For the operator \(R_{n,d}^{[N]}\) defined in \eqref{e026}, the ball average
is an \(\ell^1\) contraction, so \eqref{e153} gives
\begin{equation}\label{e157}
 \norm{R_{n,d}^{[N]}}_{\ell^1\to\ell^1}\lesssim_{A,N}1.
\end{equation}
By Plancherel's theorem and \eqref{e152},
\begin{equation}\label{e158}
 \norm{R_{n,d}^{[N]}}_{\ell^2\to\ell^2}
 \lesssim_{A,N}L^{-(N+1)/2}.
\end{equation}
Riesz--Thorin interpolation between \eqref{e157} and \eqref{e158}, with
interpolation parameter \(2(p-1)/p\), yields
\begin{equation}\label{e159}
 \norm{R_{n,d}^{[N]}f}_p
 \lesssim_{A,N,p}L^{-(N+1)(p-1)/p}\norm f_p.
\end{equation}

The remainder estimates can now be summed directly.  Since
\(\mathcal N_{d,N}\) is finite, for every \(x\in\Z^d\),
\begin{equation}
 \sup_{n\in\mathcal N_{d,N}}\abs{R_{n,d}^{[N]}f(x)}^p
 \leq\sum_{n\in\mathcal N_{d,N}}\abs{R_{n,d}^{[N]}f(x)}^p.
\label{e161}
\end{equation}
Summing \eqref{e161} over \(x\), applying \eqref{e159}, and splitting the
sum at \(n=d\), we obtain
\begin{align}
 \norm{\sup_{n\in\mathcal N_{d,N}}
 \abs{R_{n,d}^{[N]}f}}_p^p
 &\leq\sum_{n\in\mathcal N_{d,N}}
 \norm{R_{n,d}^{[N]}f}_p^p\notag\\
 &\lesssim_{A,N,p}
 \left(\sum_{L_0\leq n\leq d}n^{-\alpha}
 +\sum_{d<n\leq Ad^2}d^{-\alpha}\right)\norm f_p^p\notag\\
 &\lesssim_{A,N,p}\norm f_p^p,
\label{e160}
\end{align}
where the final inequality in \eqref{e160} follows from \eqref{e156}.  Combining
\eqref{e154} and \eqref{e160} controls all middle-range indices in
\(\mathcal N_{d,N}\).  The arbitrary expansion order is required at this
step: a fixed expansion order would not make the second sum in
\eqref{e160} converge uniformly when \(p\) is close to one.

For the indices with \(n\leq Ad^2\) excluded by \(L<L_0\), no
asymptotic estimate is needed.  If \(d\geq L_0\), then necessarily
\(n<L_0\), so there are at most \(L_0\) such averages, including the one
with \(n=0\).  If
\(d<L_0\), then there are at most \(AL_0^2+1\) indices with
\(n\leq Ad^2\).
Each is an \(\ell^p\)-contraction, and the pointwise \(p\)-sum bounds their
maximum with a constant depending only on \(A\), \(p\), and \(L_0\).
The average corresponding to \(n=0\) is the identity.

If \(n>Ad^2\), then \(\sqrt n>C_0d\) by \eqref{e109}, so
the large radius estimate \eqref{e108} applies.  Finally, for every
\(t\geq0\),
\begin{equation*}
 tB=\sqrt{\lfloor t^2\rfloor}\,B.
\end{equation*}
Combining the middle-range estimate, the finite exceptional range, and
Theorem~\ref{larg}, we obtain \eqref{e004} for \(1<p\leq2\).  Since the maximal
operator is bounded on \(\ell^\infty\) with
norm one, an interpolation leads to  \eqref{e004} for every $1<p\le\infty$ immediately. Hence the proof of Theorem~\ref{mthm} is completed.
\end{proof}

\pdfbookmark[1]{Acknowledgements}{acks}
\section*{Acknowledgements}
The author would like to thank 
Professor B{\l}a\.zej Wr\'obel for his candid and constructive
correspondence, and for his valuable suggestions on clarifying the relation
between this paper and concurrent work.

\pdfbookmark[1]{Declaration of Artificial Intelligence}{DAI}
\section*{Declaration of Artificial Intelligence}

The proof strategy in this paper was developed through repeated dialogues between the author and OpenAI's ChatGPT and Codex. The author initially asked the system to explore possible improvements of the results of \cite{NW26}, whose maximal estimates were restricted to $p\geq 2$. After several iterations, the system suggested treating the entire difficult range $n\leq Ad$ directly, and the author adopted this suggestion. To simplify the task, the author then instructed the system not to proceed through the spherical maximal operator, but instead to prove the ball maximal estimate directly. This instruction appears to have led GPT to move beyond the approach of \cite{NW26} and then to produce a new saddle point equation \eqref{e008} tailored to the ball problem. Incidentally, its role lies in  extracting the principal term for the oscillatory integral representation \eqref{e097} in a more efficient way. Further iterations refined the multiplier estimates and produced a complete argument for $p\geq2$, including the case $p=2$ posed by Stein. The author subsequently examined the generated argument independently. During this verification, the author observed that the method might also apply when $1<p<2$. An early version of Proposition \ref{hmul} retained only the first two terms in \eqref{e150} (corresponding to $N=1$). Through further interaction on generalizing the the result to $1<p<2$, the system extended that version to the expansion of arbitrary prescribed order $N$. Noe that the generating new asymptotic expansion is the key to yield the result for $1<p<2$.

On the other hand, the preprint of Jin and Su \cite{JS26}, which proves the same full-radius ball maximal theorem, appeared shortly before the first version of this paper. Approximately seven hours after that version was posted, Hormozi, Niksiński, and Wróbel  posted their work \cite{HNW26}. Professor B{\l}a\.zej Wr\'obel subsequently informed the author that their ball result had been obtained even earlier, with AI tools being used only to improve the language and presentation. Their paper also establishes a dimension-free theorem for the spherical maximal operator; see \cite[Theorem~2]{HNW26}. As explained in Subsection \ref{smeth}, the present proof has substantial methodological overlap with that of \cite{HNW26}. The author has therefore decided not to submit the present manuscript for a journal publication. This revised arXiv version is posted solely to correct and clarify the public record. The author makes no claim of priority for any result or method common to the two papers and acknowledges the prior development of those results and methods by  Hormozi,  Niksi\'nski and  Wr\'obel. Moreover, the concurrent contribution of Jin and Su \cite{JS26} is also acknowledged and compared in Subsection \ref{smeth}.

\pdfbookmark[1]{References}{refs}

\bigskip
\medskip\noindent
Sheng-Chen Mao \\
School of Mathematics and Statistics \\
Lanzhou University \\
No. 222 Tianshui South Road \\
Lanzhou 730000, P.R. China \\
\medskip\noindent
\begin{tabular}{@{}ll@{}}
{\textit{E-mail addresses:}}&{\ttfamily maoshengchen@lzu.edu.cn; maosci@163.com}
\end{tabular}

\begin{thebibliography}{99}

\bibitem{Bou86}
J.~Bourgain,
\emph{On high-dimensional maximal functions associated to convex bodies},
Amer. J. Math. \textbf{108} (1986), no.~6, 1467--1476,
\url{https://doi.org/10.2307/2374532}.

\bibitem{BouLp}
J.~Bourgain,
\emph{On the \(L^p\)-bounds for maximal functions associated to convex bodies
in \(\R^n\)},
Israel J. Math. \textbf{54} (1986), no.~3, 257--265,
\url{https://doi.org/10.1007/BF02764955}.

\bibitem{BMS19}
J.~Bourgain, M.~Mirek, E.~M. Stein, and B.~Wr\'obel,
\emph{Dimension-free estimates for discrete Hardy--Littlewood averaging
operators over the cubes in \(\Z^d\)},
Amer. J. Math. \textbf{141} (2019), no.~3, 857--905,
\url{https://doi.org/10.1353/ajm.2019.0023}.

\bibitem{BMS20}
J.~Bourgain, M.~Mirek, E.~M. Stein, and B.~Wr\'obel,
\emph{On Discrete Hardy--Littlewood Maximal Functions over the Balls in
\(\Z^d\): Dimension-Free Estimates},
in: \emph{Geometric Aspects of Functional Analysis. Vol. I},
Lecture Notes in Mathematics, vol.~2256, Springer, Cham, 2020,
127--169,
\url{https://doi.org/10.1007/978-3-030-36020-7_8}.

\bibitem{BMS21}
J.~Bourgain, M.~Mirek, E.~M. Stein, and B.~Wr\'obel,
\emph{On the Hardy--Littlewood Maximal Functions in High Dimensions:
Continuous and Discrete Perspective},
in: \emph{Geometric Aspects of Harmonic Analysis},
Springer INdAM Series, vol.~45, Springer, Cham, 2021, 107--148,
\url{https://doi.org/10.1007/978-3-030-72058-2_3}.

\bibitem{Car86}
A.~Carbery,
\emph{An almost-orthogonality principle with applications to maximal
functions associated to convex bodies},
Bull. Amer. Math. Soc. (N.S.) \textbf{14} (1986), no.~2, 269--273,
\url{https://doi.org/10.1090/S0273-0979-1986-15436-4}.

\bibitem{D26}
M.~Dymowski,
\emph{Lattice points in high-dimensional \(\ell^q\) balls with small radii},
preprint (2026), arXiv:2608.28249,
\url{https://doi.org/10.48550/arXiv.2608.28249}.

\bibitem{HNW26}
M.~Hormozi, J.~Niksi\'nski, and B.~Wr\'obel,
\emph{Dimension-free estimates for full discrete maximal functions associated
with Euclidean balls and spheres},
preprint (2026), arXiv:2609.10763v1,
\url{https://arxiv.org/abs/2609.10763v1}.

\bibitem{JS26}
K.~Jin and Q.~Su,
\emph{Full-radius dimension-free maximal inequalities for discrete Euclidean
balls},
preprint (2026), arXiv:2609.08433,
\url{https://doi.org/10.48550/arXiv.2609.08433}.

\bibitem{KMP23}
D.~Kosz, M.~Mirek, P.~Plewa, and B.~Wr\'obel,
\emph{Some remarks on dimension-free estimates for the discrete
Hardy--Littlewood maximal functions},
Israel J. Math. \textbf{254} (2023), 1--38,
\url{https://doi.org/10.1007/s11856-022-2382-7}.

\bibitem{MO90}
J.~E. Mazo and A.~M. Odlyzko,
\emph{Lattice points in high-dimensional spheres},
Monatsh. Math. \textbf{110} (1990), 47--61,
\url{https://doi.org/10.1007/BF01571276}.

\bibitem{MSW24}
M.~Mirek, T.~Z. Szarek, and B.~Wr\'obel,
\emph{Dimension-Free Estimates for the Discrete Spherical Maximal
Functions},
Int. Math. Res. Not. IMRN \textbf{2024} (2024), no.~2, 901--963,
\url{https://doi.org/10.1093/imrn/rnac329}.

\bibitem{MSW25}
M.~Mirek, T.~Z. Szarek, and B.~Wr\'obel,
\emph{Dimension-free estimates for discrete maximal functions related to
normalized Gaussians},
preprint (2025), arXiv:2503.11259,
\url{https://arxiv.org/abs/2503.11259}.

\bibitem{Mul90}
D.~M\"uller,
\emph{A geometric bound for maximal functions associated to convex bodies},
Pacific J. Math. \textbf{142} (1990), no.~2, 297--312,
\url{https://doi.org/10.2140/pjm.1990.142.297}.

\bibitem{N26}
J.~Niksi\'nski,
\emph{High-dimensional discrete 1-symmetric convex bodies and dimension-free
estimates for maximal functions},
preprint (2026), arXiv:2608.17302,
\url{https://doi.org/10.48550/arXiv.2608.17302}.

\bibitem{NW26}
J.~Niksi\'nski and B.~Wr\'obel,
\emph{Dimension-free estimates for discrete maximal functions and
lattice points in high-dimensional spheres and balls with small radii},
J. Math. Pures Appl. \textbf{214} (2026), article 103955,
\url{https://doi.org/10.1016/j.matpur.2026.103955}.

\bibitem{RW10}
W.~P. Reinhardt and P.~L. Walker,
\emph{Theta Functions}, in F.~W. J. Olver, D.~W. Lozier,
R.~F. Boisvert, and C.~W. Clark (eds.),
\emph{NIST Handbook of Mathematical Functions},
Cambridge University Press, Cambridge, 2010, Chapter~20, 523--536,
\url{https://dlmf.nist.gov/20}.

\bibitem{Ste82}
E.~M. Stein,
\emph{The development of square functions in the work of A. Zygmund},
Bull. Amer. Math. Soc. (N.S.) \textbf{7} (1982), no.~2, 359--376,
\url{https://projecteuclid.org/euclid.bams/1183549638}.

\bibitem{SS83}
E.~M. Stein and J.-O. Str\"omberg,
\emph{Behavior of maximal functions in \(\R^n\) for large \(n\)},
Ark. Mat. \textbf{21} (1983), no.~2, 259--269,
\url{https://doi.org/10.1007/BF02384314}.

\end{thebibliography}
\end{document}